\documentclass{aims-mod}

\usepackage{amsmath}
\usepackage{paralist}
\usepackage{graphicx}               
\usepackage{epsfig}                 
\usepackage{epstopdf}
\usepackage[colorlinks=true,hyperfootnotes=false]{hyperref} 
\usepackage[figure,table]{hypcap}   

\usepackage[nocompress]{cite}       
\usepackage{amssymb}                
\usepackage{booktabs}               
\usepackage{threeparttable}         
\usepackage[normalem]{ulem}         
\usepackage{comment}
\usepackage[symbol]{footmisc}       
\DefineFNsymbols{nostar}{           
{**}{***}
\dagger{\dagger\dagger}
{\ddagger}{\ddagger\ddagger}}
\setfnsymbol{nostar}

\usepackage{multirow,bigdelim}

\newcommand{\todo}[1]{\textcolor{blue}{#1}}
\newcommand{\important}[1]{\textcolor{red}{#1}}
\newcommand{\editorial}[1]{\textcolor{blue}{[#1]}}
\newcommand{\pinv}{+}
\DeclareMathOperator{\trace}{Tr}
\newcommand{\nullspace}{\mathcal{N}}
\newcommand{\range}{\mathcal{R}}
\let\phi\varphi

\theoremstyle{definition}
\newtheorem{definition}{Definition}
\newtheorem{remark}{Remark}
\newtheorem{algorithm}{Algorithm}
\theoremstyle{plain}
\newtheorem{theorem}{Theorem}

\usepackage{ifthen}
\newcommand{\switch}[1]{\ifthenelse{\equal{#1}{0}}{
    \renewcommand{\todo}[1]{}
    \renewcommand{\important}[1]{}
    \renewcommand{\editorial}[1]{[##1]}
    }{}}
\switch{0} 

\hypersetup{urlcolor=blue, citecolor=red}
\def\currentvolume{\editorial{X}}
\def\currentissue{\editorial{X}}
\def\currentyear{\editorial{XXXX}}
\def\currentmonth{\editorial{XX}}
\def\ppages{\editorial{XX--XX}}

\title[On dynamic mode decomposition]{On dynamic mode decomposition: \\ theory and applications}

\author[J. H. Tu et al.]{}



\email{jhtu@berkeley.edu}
\email{cwrowley@princeton.edu}
\email{dluchten@princeton.edu}
\email{sbrunton@uw.edu}
\email{kutz@uw.edu}

\begin{document}
\maketitle

\centerline{\scshape Jonathan H. Tu, Clarence W. Rowley, Dirk M. Luchtenburg,}
\medskip
{\footnotesize
\centerline{Dept.\ of Mechanical and Aerospace Engineering}
\centerline{Princeton University}
\centerline{Princeton, NJ 08544, USA}
}

\medskip

\centerline{\scshape Steven L. Brunton, and J. Nathan Kutz}
\medskip

{\footnotesize
\centerline{Dept.\ of Applied Mathematics}
\centerline{University of Washington}
\centerline{Seattle, WA 98195, USA}
}

\bigskip





\vspace{12pt}
\important{
TODO items:
\begin{itemize}
    \item Re-run examples or verify that for this data exact and projected DMD are the same.
    \item Add notes in the examples about what scaling is used.
\end{itemize}
}


\begin{abstract}
    Originally introduced in the fluid mechanics community, dynamic mode decomposition (DMD) has emerged as a powerful tool for analyzing the dynamics of nonlinear systems.
    However, existing DMD theory deals primarily with sequential time series for which the measurement dimension is much larger than the number of measurements taken.
    We present a theoretical framework in which we define DMD as the eigendecomposition of an approximating linear operator.
    This generalizes DMD to a larger class of datasets, including nonsequential time series.
    We demonstrate the utility of this approach by presenting novel sampling strategies that increase computational efficiency and mitigate the effects of noise, respectively.
    We also introduce the concept of linear consistency, which helps explain the potential pitfalls of applying DMD to rank-deficient datasets, illustrating with examples.
    Such computations are not considered in the existing literature, but can be understood using our more general framework.
    In addition, we show that our theory strengthens the connections between DMD and Koopman operator theory.
    It also establishes connections between DMD and other techniques, including the eigensystem realization algorithm (ERA), a system identification method, and linear inverse modeling (LIM), a method from climate science.
    We show that under certain conditions, DMD is equivalent to LIM.
\end{abstract}


\section{Introduction}
\label{sec:intro}

Fluid flows often exhibit low-dimensional behavior, despite the fact that they are governed by infinite-dimensional partial differential equations (the Navier--Stokes equations).
For instance, the main features of the laminar flow past a two-dimensional cylinder can be described using as few as three ordinary differential equations~\cite{noackJFM03}.
To identify these low-order dynamics, such flows are often analyzed using \emph{modal decomposition} techniques, including proper orthogonal decomposition (POD), balanced proper orthogonal decomposition (BPOD), and dynamic mode decomposition (DMD).
Such methods describe the fluid state (typically the velocity or vorticity field) as a superposition of empirically computed basis vectors, or ``modes.''
In practice, the number of modes necessary to capture the gross behavior of a flow is often many orders of magnitude smaller than the state dimension of the system (e.g., $\mathcal{O}(10)$ compared to $\mathcal{O}(10^6)$).

Since it was first introduced in~\cite{schmidAPS08}, DMD has quickly gained popularity in the fluids community, primarily because it provides information about the \emph{dynamics} of a flow, and is applicable even when those dynamics are \emph{nonlinear}~\cite{rowleyJFM09}.
A typical application involves collecting a time series of experimental or simulated velocity fields, and from them computing DMD modes and eigenvalues.
The modes are spatial fields that often identify coherent structures in the flow.
The corresponding eigenvalues define growth/decay rates and oscillation frequencies for each mode.
Taken together, the DMD modes and eigenvalues describe the dynamics observed in the time series in terms of oscillatory components.
In contrast, POD modes optimally reconstruct a dataset, with the modes ranked in terms of energy content~\cite{HLBR-11}.
BPOD modes identify spatial structures that are important for capturing linear input-output dynamics~\cite{rowleyIJBC05}, and can also be interpreted as an optimal decomposition of two (dual) datasets~\cite{Singler:2010}.

At first glance, it may seem dubious that a nonlinear system could be described by superposition of modes whose dynamics are governed by eigenvalues.
After all, one needs a {\em linear} operator in order to talk about eigenvalues.
However, it was shown in~\cite{rowleyJFM09} that DMD is closely related to a spectral analysis of the Koopman operator.
The Koopman operator is a linear but infinite-dimensional operator whose modes and eigenvalues capture the evolution of observables describing any (even nonlinear) dynamical system.
The use of its spectral decomposition for data-based modal decomposition and model reduction was first proposed in~\cite{mezicND05}.
DMD analysis can be considered to be a numerical approximation to Koopman spectral analysis, and it is in this sense that DMD is applicable to nonlinear systems.
In fact, the terms ``DMD mode'' and ``Koopman mode'' are often used interchangably in the fluids literature.

Much of the recent work involving DMD has focused on its application to different flow configurations.
For instance, DMD has been applied in the study of the wake behind a flexible membrane~\cite{schmidJFM10}, the flow around high-speed trains~\cite{muldCF12}, instabilities in annular liquid sheets~\cite{dukeJFM12}, shockwave-turbulent boundary layer interactions~\cite{grilliJFM12}, detonation waves~\cite{massaPF12}, cavity flows~\cite{schmidJFM10,seenaIJHFF11}, and various jets~\cite{rowleyJFM09,schmidJFM10,schmidEF11,schmidTCFD11,schmidEF12,semeraroEF12}.
There have also been a number of efforts regarding the numerics of the DMD algorithm, including the development of memory-efficient algorithms~\cite{tuJCP12,belsonACMTMS13}, an error analysis of DMD growth rates~\cite{dukeEF12}, and a method for selecting a sparse basis of DMD modes~\cite{jovanovicArxiv13}.
Variants of the DMD algorithm have also been proposed, including optimized DMD~\cite{chenJNLS12} and optimal mode decomposition~\cite{goulartCDC12,wynnJFM12}.
Theoretical work on DMD has centered mainly on exploring connections with other methods, such as Koopman spectral analysis~\cite{mezicARFM12,rowleyJFM09,bagheriJFM13}, POD~\cite{schmidJFM10}, and Fourier analysis~\cite{chenJNLS12}.
Theorems regarding the existence and uniqueness of DMD modes and eigenvalues can be found in~\cite{chenJNLS12}.
For a review of the DMD literature, we refer the reader to~\cite{mezicARFM12}.

Many of the papers cited above mention the idea that DMD is able to characterize nonlinear dynamics through an analysis of some approximating linear system.
In this work, we build on this notion.
We present DMD as an analysis of \emph{pairs} of $n$-dimensional data vectors $(x_k, y_k)$, in contrast to the sequential time series that are typically considered.
From these data we construct a particular linear operator~$A$ and define DMD as the eigendecomposition of that operator (see Definition~\ref{def:exact-DMD}).
We show that DMD modes satisfying this definition can be computed using a modification of the algorithm proposed in~\cite{schmidJFM10}.
Both algorithms generate the same eigenvalues, with the modes differing by a projection (see Theorem~\ref{thm:proj-DMD-modes}).

There is of course no guarantee that analyzing this particular approximating operator is meaningful for data generated by nonlinear dynamics.
To this end, we show that our definition strengthens the connections between DMD and Koopman operator theory, extending those connections to include more general sampling strategies.
This is important, as it allows us to maintain the interpretion of DMD as an approximation to Koopman spectral analysis.
We can then be confident that DMD is useful for characterizing nonlinear dynamics.
Furthermore, we show that the connections between DMD and Koopman spectral analysis hold not only when the vectors $x_k$ are linearly independent, as assumed in~\cite{rowleyJFM09}, but under a more general condition we refer to as linear consistency (see Definition~\ref{def:linear-consistency}).
When the data are not linearly consistent, the Koopman analogy can break down and DMD analysis may produce either meaningful or misleading results.
We show an example of each and explain the results using our approximating-operator definition of DMD.
(For a more detailed investigation of how well DMD eigenvalues approximate Koopman eigenvalues, we refer the reader to~\cite{bagheriJFM13}.)

The generality of our framework has important practical implications as well.
To this end, we present examples demonstrating the benefits of applying DMD to nonsequential time series.
For instance, we show that nonuniform temporal sampling can provide increased computational efficiency, with little effect on accuracy of the dominant DMD modes and eigenvalues.
We also show that noise in experimental datasets can be dealt with by concatenating data from multiple runs of an experiment.
The resulting DMD computation produces a spectrum with sharper, more isolated peaks, allowing us to identify higher-frequency modes that are obscured in a traditional DMD computation.

Finally, our framework highlights the connections between DMD and other well-known methods, specifically the eigensystem realization algorithm (ERA) and linear inverse modeling (LIM).
The ERA is a control-theoretic method for system identification of linear systems~\cite{hoACCST65,juangJGCD85,maTCFD11}.
We show that when computed from the same data, DMD eigenvalues reduce to poles of an ERA model.
This connection motivates the use of ERA-inspired strategies for dealing with certain limitations of DMD.
LIM is a modeling procedure developed in the climate science community~\cite{penlandMWR89,penlandJClimate93}.
We show that under certain conditions, DMD is equivalent to LIM.
Thus it stands to reason that practioners of DMD could benefit from an awareness of related work in the climate science literature.

The remainder of this work is organized as follows: in Section~\ref{sec:theory}, we propose and discuss a new definition of DMD.
We provide several different algorithms for computing DMD modes and eigenvalues that satisfy this new definition and show that these are closely related to the modes and eigenvalues computed using the currently accepted SVD-based DMD algorithm~\cite{schmidJFM10}.
A number of examples are presented in Section~\ref{sec:applications}.
These explore the application of DMD to rank-deficient datasets and nonsequential time series.
Section~\ref{sec:connections} describes the connections between DMD and Koopman operator theory, the ERA, and LIM, respectively.
We summarize our results in Section~\ref{sec:concl}.


\section{Theory}
\label{sec:theory}

Since its first appearance in 2008~\cite{schmidAPS08}, DMD has been defined by an algorithm (the specifics of which are given in Algorithm~\ref{alg:proj-DMD} below).
Here, we present a more general, non-algorithmic definition of DMD.
Our definition emphasizes data that are collected as a set of \emph{pairs} $\{(x_k,y_k)\}_{k=1}^m$, rather than as a sequential time series $\{z_k\}_{k=0}^m$.
We show that our DMD definition and algorithm are closely related to the currently accepted algorithmic definition.
In fact, the two approaches produce the same DMD eigenvalues; it is only the DMD modes that differ.

\subsection{Standard definition}
\label{sec:standard-def}

Originally, the DMD algorithm was formulated in terms of a companion matrix~\cite{schmidAPS08,rowleyJFM09}, which highlights its connections to the Arnoldi algorithm and to Koopman operator theory.
The SVD-based algorithm presented in~\cite{schmidJFM10} is more numerically stable, and is now generally accepted as the defining DMD algorithm; we describe this algorithm below.

Consider a sequential set of data vectors $\{z_0,\ldots,z_m\}$, where each $z_k \in \mathbb{R}^n$.
We assume that the data are generated by linear dynamics
\begin{equation}
    \label{eq:lin-dyn}
    z_{k+1} = A z_k,
\end{equation}
for some (unknown) matrix~$A$.
(Alternatively, the vectors $z_k$ can be sampled from a continuous evolution $z(t)$, in which case $z_k = z(k \Delta t)$ and a fixed sampling rate $\Delta t$ is assumed.)
When DMD is applied to data generated by nonlinear dynamics, it is assumed that there exists an operator~$A$ that approximates those dynamics.
The DMD modes and eigenvalues are intended to approximate the eigenvectors and eigenvalues of~$A$.
The algorithm proceeds as follows:

\begin{algorithm}[Standard DMD]
    \label{alg:proj-DMD}
    \mbox{}

    \begin{enumerate}
        \item Arrange the data $\{z_0,\ldots,z_m\}$ into matrices
        \begin{equation}
            \label{eq:def-XY-sequential}
            X \triangleq \begin{bmatrix} z_0 & \cdots & z_{m-1} \end{bmatrix},
            \qquad
            Y \triangleq \begin{bmatrix} z_1 & \cdots & z_m \end{bmatrix}.
        \end{equation}

        \item Compute the (reduced) SVD of $X$ (see~\cite{trefethenNLA97}), writing
        \begin{equation}
            \label{svd-of-X}
            X = U \Sigma V^*,
        \end{equation}
        where $U$ is $n \times r$, $\Sigma$ is diagonal and $r \times r$, $V$ is $m \times r$, and~$r$ is the rank of~$X$.

        \item Define the matrix
        \begin{equation}
            \label{eq:def-Atilde}
            \tilde A \triangleq U^* Y V \Sigma^{-1}.
        \end{equation}

        \item Compute eigenvalues and eigenvectors of $\tilde A$, writing
        \begin{equation}
            \label{eq:low-order-eig-prob}
            \tilde A w = \lambda w.
        \end{equation}

        \item The DMD mode corresponding to the DMD eigenvalue~$\lambda$ is then given by
        \begin{equation}
            \label{eq:def-proj-DMD-mode}
            \hat \phi \triangleq U w.
        \end{equation}

        \item If desired, the DMD modes can be scaled in a number of ways, as described in Appendix~\ref{sec:mode-scalings}.
    \end{enumerate}
\end{algorithm}
In this paper, we will refer to the modes produced by Algorithm~\ref{alg:proj-DMD} as {\em projected DMD modes}, for reasons that will become apparent in Section~\ref{ssec:compare-defs} (in particular, see Theorem~\ref{thm:proj-DMD-modes}).

\subsection{New definition}
\label{ssec:new-def}

The standard definition of DMD assumes a sequential set of data vectors $\{z_0,\ldots,z_m\}$ in which the order of vectors $z_k$ is critical.
Furthermore, the vectors should (at least approximately) satisfy the relation~\eqref{eq:lin-dyn}.
Here, we relax these restrictions on the data, and consider data \emph{pairs} $\{(x_1,y_1),\ldots,(x_m,y_m)\}$.
We then define DMD in terms of the $n \times m$ data matrices
\begin{equation}
    \label{eq:def-XY-general}
    X \triangleq \begin{bmatrix} x_1 & \cdots & x_m \end{bmatrix},
    \qquad
    Y \triangleq \begin{bmatrix} y_1 & \cdots & y_m \end{bmatrix}.
\end{equation}
Note that the formulation~\eqref{eq:def-XY-sequential} is a special case of~\eqref{eq:def-XY-general}, with $x_k = z_{k-1}$ and $y_k = z_k$.
In order to relate this method to the standard DMD procedure, we may assume that
\begin{equation*}
    y_k = \hat A x_k
\end{equation*}
for some (unknown) matrix~$\hat A$.
However, the procedure here is applicable more generally.
We define the DMD modes of this dataset as follows:

\begin{definition}[Exact DMD]
    \label{def:exact-DMD}
    For a dataset given by~\eqref{eq:def-XY-general}, define the operator
    \begin{equation}
        \label{eq:def-A}
        A \triangleq Y X^\pinv,
    \end{equation}
    where $X^\pinv$ is the pseudoinverse of $X$.
    The \emph{dynamic mode decomposition} of the pair $(X, Y)$ is given by the eigendecomposition of $A$.
    That is, the DMD modes and eigenvalues are the eigenvectors and eigenvalues of $A$.
\end{definition}

\begin{remark}
    \label{rmk:best-fit-operator}
    The operator $A$ in~\eqref{eq:def-A} is the least-squares/minimum-norm solution to the potentially over- or under-constrained problem $A X = Y$.
    That is, if there is an exact solution to $A X = Y$ (which is always the case if the vectors $x_k$ are linearly independent), then the choice~\eqref{eq:def-A} minimizes $\|A\|_F$, where $\|A\|_F=\trace(AA^*)^{1/2}$ denotes the Frobenius norm.
    If there is no $A$ that exactly satisfies $AX=Y$, then the choice~\eqref{eq:def-A} minimizes $\|AX-Y\|_F$.
\end{remark}

When $n$ is large, as is often the case with fluid flow data, it may be inefficient to compute the eigendecomposition of the $n \times n$ matrix $A$.
In some cases, even storing $A$ in memory can be prohibitive.
Using the following algorithm, the DMD modes and eigenvalues can be computed without an explicit representation or direct manipulations of $A$.

\begin{algorithm}[Exact DMD]
    \label{alg:exact-DMD}
  \mbox{}

  \begin{enumerate}
      \item Arrange the data pairs $\{(x_1,y_1),\ldots,(x_m,y_m)\}$ into matrices $X$ and~$Y$, as in~\eqref{eq:def-XY-general}.

      \item Compute the (reduced) SVD of $X$, writing $X = U \Sigma V^*$.

      \item Define the matrix $\tilde A \triangleq U^* Y V \Sigma^{-1}$.

      \item Compute eigenvalues and eigenvectors of $\tilde A$, writing $\tilde A w = \lambda w$.
      Each nonzero eigenvalue~$\lambda$ is a DMD eigenvalue.

      \item The DMD mode corresponding to~$\lambda$ is then given by
      \begin{equation}
          \label{eq:def-exact-DMD-mode}
          \phi \triangleq \frac{1}{\lambda} Y V \Sigma^{-1} w.
      \end{equation}

      \item If desired, the DMD modes can be scaled in a number of ways, as described in Appendix~\ref{sec:mode-scalings}.
  \end{enumerate}
\end{algorithm}

\begin{remark}
    When $n \gg m$, the above algorithm can be modified to reduce computational costs.
    For instance, the SVD of $X$ can be computed efficiently using the method of snapshots~\cite{sirovichQAM87_2}.
    This involves computing the correlation matrix $X^* X$.
    The product $U^* Y$ required to form $\tilde{A}$ can be cast in terms of a product $X^* Y$, using~\eqref{svd-of-X} to substitute for $U$.
    If $X$ and $Y$ happen to share columns, as is the case for sequential time series, then $X^* Y$ will share elements with $X^* X$, reducing the number of new computations required.
    (See Section~\ref{ssec:nonseq-sampling} for more on sequential versus nonsequential time series.)
\end{remark}

Algorithm~\ref{alg:exact-DMD} is nearly identical to Algorithm~\ref{alg:proj-DMD} (originally presented in~\cite{schmidJFM10}).
In fact, the only difference is that the DMD modes are given by~\eqref{eq:def-exact-DMD-mode}, whereas in Algorithm~\ref{alg:proj-DMD}, they are given by~\eqref{eq:def-proj-DMD-mode}.
This modification is subtle, but important, as we discuss in Section~\ref{ssec:compare-defs}.

\begin{remark}
    Though the original presentations of DMD~\cite{rowleyJFM09,schmidJFM10} assume $X$ and $Y$ of the form given by~\eqref{eq:def-XY-sequential}, Algorithm~\ref{alg:proj-DMD} does not make use of this structure.
    That is, the algorithm can be carried out for the general case of $X$ and $Y$ given by~\eqref{eq:def-XY-general}.
    (This point has been noted previously, in~\cite{wynnJFM12}.)
\end{remark}

\begin{theorem}
    \label{thm:alg-yields-eigdecomp-of-A}
    Each pair $(\phi, \lambda)$ generated by Algorithm~\ref{alg:exact-DMD} is an eigenvector/eigenvalue pair of~$A$.  Furthermore, the algorithm identifies all of the nonzero eigenvalues of~$A$.
\end{theorem}
\begin{proof}
    From the SVD $X = U \Sigma V^*$, we may write the pseudoinverse of $X$ as
    \begin{equation*}
        X^\pinv = V \Sigma^{-1} U^*,
    \end{equation*}
    so from~\eqref{eq:def-A}, we find
    \begin{equation}
        \label{eq:A-as-SVD-of-X}
        A = Y V \Sigma^{-1} U^* = B U^*,
    \end{equation}
    where
    \begin{equation}
        \label{eq:def-B}
        B \triangleq Y V \Sigma^{-1}.
    \end{equation}
    In addition, we can rewrite~\eqref{eq:def-Atilde} as
    \begin{equation}
        \label{eq:Atilde-as-B}
        \tilde A = U^* Y V \Sigma^{-1} = U^* B.
    \end{equation}
    Now, suppose that $\tilde A w = \lambda w$, with $\lambda \ne 0$, and let $\phi= \frac{1}{\lambda} B w$, as in~\eqref{eq:def-exact-DMD-mode}.
    Then
    \begin{equation*}
        A \phi = \frac{1}{\lambda} B U^* B w = B \frac{1}{\lambda} \tilde A w = B w = \lambda \phi.
    \end{equation*}
    In addition, $\phi\ne 0$, since if $Bw=0$, then $U^*Bw=\tilde A w=0$,
    so $\lambda=0$.  Hence, $\phi$ is an eigenvector of~$A$ with
    eigenvalue~$\lambda$.

    To show that Algorithm~\ref{alg:exact-DMD} identifies all of the nonzero eigenvalues of~$A$, suppose $A \phi = \lambda \phi$, for $\lambda \ne 0$, and let $w = U^* \phi$.  Then
    \begin{equation*}
      \tilde A w = U^*BU^*\phi = U^* A\phi = \lambda U^*\phi = \lambda w.
    \end{equation*}
    Furthermore, $w\ne 0$, since if $U^*\phi=0$, then $BU^*\phi=A\phi=0$, and $\lambda=0$.
    Thus, $w$ is an eigenvector of~$\tilde A$ with eigenvalue~$\lambda$, and is identified by Algorithm~\ref{alg:exact-DMD}.
\end{proof}

\begin{remark}
    Algorithm~\ref{alg:exact-DMD} may also be used to find certain eigenvectors with $\lambda=0$ (that is, in the nullspace of~$A$).
    In particular, if $\tilde{A} w = 0$ and $\phi = Y V \Sigma^{-1} w \neq 0$, then $\phi$ is an eigenvector with $\lambda=0$ (and is in the image of $Y$); if $\tilde{A} w = 0$ and $Y V \Sigma^{-1} w = 0$, then $\phi = U w$ is an eigenvector with $\lambda=0$ (and is in the image of $X$).
    However, DMD modes corresponding to zero eigenvalues are usually not of interest, since they do not play a role in the dynamics.
\end{remark}

Next, we characterize the conditions under which the operator~$A$ defined by~\eqref{eq:def-A} satisfies $Y = A X$.
We emphasize that this does not require that~$Y$ is generated from~$X$ through linear dynamics defined by~$A$; we place no restrictions on the data pairs $(x_k,y_k)$.
To this end, we introduce the following definition:

\begin{definition}[Linear consistency]
    \label{def:linear-consistency}
    Two $n \times m$ matrices $X$ and $Y$ are {\em linearly consistent} if, whenever $X c = 0$, then $Y c = 0$ as well.
\end{definition}

Thus $X$ and $Y$ are linearly consistent if and only if the nullspace of $Y$ contains the nullspace of $X$.
If the vectors~$x_k$ (columns of~$X$) are linearly independent, the nullspace of~$X$ is~$\{0\}$ and linear consistency is satisfied trivially.
However, linear consistency does not imply that the columns of~$X$ are linearly independent.
(Linear consistency will play an important role in establishing the connection between DMD modes and Koopman modes, as discussed in Section~\ref{ssec:Koopman}.)

The notion of linear consistency makes intuitive sense if we think of the vectors~$x_k$ as inputs and the vectors~$y_k$ as outputs.
Definition~\ref{def:linear-consistency} follows from the idea that two identical inputs should not have different outputs, generalizing the idea to arbitrary linearly dependent sets of inputs.
It turns out that for linearly consistent data, the approximating operator $A$ relates the datasets exactly ($A X = Y$), even if the data are generated by nonlinear dynamics.

\begin{theorem}
    \label{thm:AX-eq-Y}
    Define $A = Y X^\pinv$, as in Definition~\ref{def:exact-DMD}.
    Then $Y = A X$ if and only if $X$ and~$Y$ are linearly consistent.
\end{theorem}

\begin{proof}
    First, suppose~$X$ and~$Y$ are not linearly consistent.
    Then there exists~$v$ in the nullspace of~$X$ (denoted $\nullspace(X)$) such that $Y v \ne 0$.
    But then $A X v = 0 \ne Y v$, so $A X \ne Y$ for any~$A$.

    Conversely, suppose~$X$ and~$Y$ are linearly consistent: that is, $\nullspace(X) \subset \nullspace(Y)$.
    Then
    \begin{equation*}
        Y - A X = Y - Y X^\pinv X = Y (I - X^\pinv X).
    \end{equation*}
    Now, $X^\pinv X$ is the orthogonal projection onto the range of $X^*$ (denoted $\range(X^*)$), so $I - X^\pinv X$ is the orthogonal projection onto $\range(X^*)^\perp = \nullspace(X)$.
    Thus, since $\nullspace(X) \subset \nullspace(Y)$, it follows that $Y (I - X^\pinv X) = 0$, so $Y = A X$.
\end{proof}

We note that given Definition~\ref{def:exact-DMD}, is natural to define adjoint DMD modes as the eigenvectors of $A^*$ (or, equivalently, the left eigenvectors of~$A$).
Computing such modes can be done easily with slight modifications to Algorithm~\ref{alg:exact-DMD}.
Let $z$ be an adjoint eigenvector of $\tilde{A}$, so that $z^*\tilde{A} = \lambda z^*$.
Then one can easily verify that $\psi \triangleq U z$ is an adjoint eigenvector of~$A$: $\psi^*A=\lambda\psi^*$.
It is interesting to note that while the DMD modes corresponding to nonzero eigenvalues lie in the image of $Y$, the adjoint DMD modes lie in the image of $X$.

\subsection{Comparing definitions}
\label{ssec:compare-defs}

Algorithm~\ref{alg:proj-DMD} (originally presented in~\cite{schmidJFM10}) has come to dominate among DMD practioners due to its numerical stability.
Effectively, it has become the working definition of DMD.
As mentioned above, it differs from Algorithm~\ref{alg:exact-DMD} only in that the (projected) DMD modes are given by
\begin{equation*}
    \hat \phi \triangleq U w,
\end{equation*}
where $w$ is an eigenvector of $\tilde{A}$, while the exact modes DMD modes are given by~\eqref{eq:def-exact-DMD-mode} as
\begin{equation*}
    \phi \triangleq \frac{1}{\lambda} Y V \Sigma^{-1} w.
\end{equation*}

Since $U$ contains left singular vectors of $X$, we see that the original modes defined by~\eqref{eq:def-proj-DMD-mode} lie in the image of $X$, while those defined by~\eqref{eq:def-exact-DMD-mode} lie in the image of $Y$.
As a result, the modes~$\hat\phi$ are not eigenvectors of the approximating linear operator~$A$.
Are they related in any way to the eigenvectors of~$A$?
The following theorem establishes that they are, and motivates the terminology {\em projected} DMD modes.

\begin{theorem}
    \label{thm:proj-DMD-modes}
    Let $\tilde A w = \lambda w$, with $\lambda \ne 0$, and let $\mathbb{P}_X$ denote the orthogonal projection onto the image of~$X$.
    Then $\hat{\phi}$ given by~\eqref{eq:def-proj-DMD-mode} is an eigenvector of $\mathbb{P}_X A$ with eigenvalue~$\lambda$.
    Furthermore, if $\phi$ is given by~\eqref{eq:def-exact-DMD-mode}, then $\hat \phi = \mathbb{P}_X \phi$.
\end{theorem}

\begin{proof}
    From the SVD $X=U\Sigma V^*$, the orthogonal projection onto the image of $X$ is given by $\mathbb{P}_X = U U^*$.
    Using~\eqref{eq:A-as-SVD-of-X} and~\eqref{eq:Atilde-as-B}, and recalling that $U^* U$
    is the identity matrix, we have
    \begin{align*}
        \mathbb{P}_X A \hat{\phi} &= (U U^*) (B U^*) (U w) = U (U^* B) w \\
        &= U\tilde A w = \lambda U w = \lambda \hat \phi.
    \end{align*}
    Then $\hat \phi$ is an eigenvector of $\mathbb{P}_X A$ with eigenvalue~$\lambda$.
    Moreover, if $\phi$ is given by~\eqref{eq:def-exact-DMD-mode}, then
    \begin{equation*}
        U^* \phi = \frac{1}{\lambda} U^* B w = \frac{1}{\lambda} \tilde A w = w,
    \end{equation*}
    so $\mathbb{P}_X \phi = U U^* \phi = U w = \hat\phi$.
\end{proof}
Thus, the modes $\hat\phi$ determined by Algorithm~\ref{alg:proj-DMD} are simply the projection of the modes determined by Algorithm~\ref{alg:exact-DMD} onto the range of~$X$.
Also, note that $U^*\phi = U^*\hat\phi = w$.

\begin{remark}
  \label{rem:exact-proj-equiv}
  If the vectors $\{y_k\}$ lie in the span of the vectors~$\{x_k\}$, then $\mathbb{P}_XA = A$, and the projected and exact DMD modes are identical.
  For instance, this is the case for a sequential time series (as in~\eqref{eq:def-XY-sequential}) when the last vector~$z_m$ is a linear combination of $\{z_0,\ldots,z_{m-1}\}$.
\end{remark}

In addition to providing an improved algorithm,~\cite{schmidJFM10}~also explores the connection between (projected) DMD and POD.
For data generated by linear dynamics $z_{k+1} = A z_{k}$,~\cite{schmidJFM10}~derives the relation $\tilde{A} = U^* A U$ and notes that we can interpret~$\tilde{A}$ as the correlation between the matrix of POD modes $U$ and the matrix of time-shifted POD modes~$A U$.
Of course,~$\tilde{A}$ also determines the DMD eigenvalues.
Exact DMD is based on the definition of $A$ as the least-squares/minimum-norm solution to $A X = Y$, but does not require that equation to be satisfied exactly.
Even so, we can combine \eqref{eq:A-as-SVD-of-X} and~\eqref{eq:Atilde-as-B} to find that $\tilde{A} = U^* A U$.
Thus exact DMD preserves the interpretation of $\tilde{A}$ in terms of POD modes, extending it from sequential time series to generic data matrices, and without making any assumptions about the dynamics relating $X$ and $Y$.


To further clarify the connections between exact and projected DMD, consider the following alternative algorithm for computing DMD eigenvalues and exact DMD modes:
\begin{algorithm}[Exact DMD, alternative method]
    \label{alg:big-svd}
    \mbox{}

    \begin{enumerate}
        \item Arrange the data pairs $\{(x_1,y_1),\ldots,(x_m,y_m)\}$ into matrices $X$ and~$Y$, as in~\eqref{eq:def-XY-general}.

        \item Compute the (reduced) SVD of $X$, writing $X = U \Sigma V^*$.

        \item Compute an orthonormal basis for the column space of  $\begin{bmatrix} X & Y \end{bmatrix}$, stacking the basis vectors in a matrix~$Q$ such that $Q^*Q=I$.
          For example, $Q$ can be computed by singular value decomposition of $\begin{bmatrix} X & Y \end{bmatrix}$, or by a $QR$ decomposition.

        \item Define the matrix
        \begin{equation}
            \tilde{A}_{Q} \triangleq Q^*AQ,
        \end{equation}
        where $A = Y V \Sigma^{-1} U^*$, as in~\eqref{eq:A-as-SVD-of-X}.

        \item Compute eigenvalues and eigenvectors of $\tilde{A}_{Q}$, writing $\tilde{A}_{Q} w = \lambda w$.
        Each nonzero eigenvalue~$\lambda$ is a DMD eigenvalue.

        \item The DMD mode corresponding to~$\lambda$ is then given by
        \begin{equation*}
            \phi \triangleq Q w.
        \end{equation*}

        \item If desired, the DMD modes can be scaled in a number of ways, as described in Appendix~\ref{sec:mode-scalings}.
    \end{enumerate}
\end{algorithm}

To see that $\phi$ computed by Algorithm~\ref{alg:big-svd} is an exact DMD mode, we first observe that because the columns of $Q$ span the column space of $Y$, we can write $Y = QQ^* Y$, and thus $A = Q Q^* A$.
Then we find that
\begin{align*}
    A \phi &= Q Q^* A \phi = Q Q^* A Q v = Q \tilde{A}_{Q} v = \lambda Q v = \lambda \phi.
\end{align*}

We emphasize that because the above algorithm requires both an SVD of $X$ and a QR decomposition of the augmented matrix $\begin{bmatrix} X & Y \end{bmatrix}$, it is more costly than Algorithm~\ref{alg:exact-DMD}, which should typically be used in practice.
However, Algorithm~\ref{alg:big-svd} proves useful in showing how exact DMD can be interpreted as a simple extension of the standard algorithm for projected DMD (see Algorithm~\ref{alg:proj-DMD}).
Recall that in projected DMD, one constructs the matrix $\tilde A = U^* A U$, where $U$ arises from the SVD of~$X$.
One then finds eigenvectors $w$ of $\tilde A$, and the projected DMD modes have the form $\hat\phi = Uw$.
Algorithm~\ref{alg:big-svd} is precisely analogous, with $U$ replaced by~$Q$: one constructs $\tilde A_Q = Q^*AQ$, finds eigenvectors~$v$ of $\tilde A_Q$, and the exact DMD modes have the form $\phi=Qw$.
In the case of a sequential time series, where $X$ and~$Y$ are given by~\eqref{eq:def-XY-sequential}, projected DMD projects $A$ onto the space spanned by the first $m$ vectors $\{z_0,\ldots,z_{m-1}\}$ (columns of~$X$), while exact DMD projects~$A$ onto the space spanned by all $m+1$ vectors $\{z_0,\ldots,z_m\}$ (columns of~$X$ and~$Y$).
In this sense, exact DMD is perhaps more natural, as it uses all of the data, rather than leaving out the last vector.

The case of a sequential time series is so common that it deserves special attention.
For this case, we provide a variant of Algorithm~\ref{alg:big-svd} that is computationally advantageous.

\begin{algorithm}[Exact DMD, sequential time series]
    \label{alg:exact-DMD-sequential}
    \mbox{}

    \begin{enumerate}
        \item Arrange the data $\{z_0,\ldots,z_m\}$ into matrices $X$ and~$Y$, as in~\eqref{eq:def-XY-sequential}.

        \item Compute the (reduced) SVD of $X$, writing $X = U \Sigma V^*$.

        \item Compute a vector~$q$ such that the columns of
          \begin{equation}
              \label{eq:def-Q-sequential}
              Q = \begin{bmatrix} U & q \end{bmatrix}
          \end{equation}
          form an orthonormal basis for $\{z_0,\ldots,z_m\}$.
          For instance, $q$ may be computed by one step of the Gram-Schmidt procedure applied to $z_m$:
          \begin{equation}
              p = z_m - U U^* z_m,\qquad q = \frac{p}{\|p\|}.
          \end{equation}
          (If $p = 0$, then take $Q=U$; in this case, exact DMD is identical to projected DMD, as discussed in Remark~\ref{rem:exact-proj-equiv}.)

        \item Define the matrix $\tilde A\triangleq U^* Y V \Sigma^{-1}$, as in~\eqref{eq:def-Atilde}.

        \item  Compute eigenvalues and eigenvectors of $\tilde{A}$, writing $\tilde{A} w = \lambda w$.
        Each nonzero eigenvalue~$\lambda$ is a DMD eigenvalue.

        \item The DMD mode corresponding to~$\lambda$ is then given by
        \begin{equation}
            \label{eq:exact-mode-correction}
            \phi \triangleq U w + \frac{1}{\lambda} qq^* B w,
        \end{equation}
        where $B=YV\Sigma^{-1}$.

        \item If desired, the DMD modes can be scaled in a number of ways, as described in Appendix~\ref{sec:mode-scalings}.
    \end{enumerate}
\end{algorithm}

Here, \eqref{eq:exact-mode-correction} is obtained from~\eqref{eq:def-exact-DMD-mode}, noting that $B=QQ^*B$, so
\begin{equation*}
    Bw = QQ^*Bw = UU^*Bw + qq^*Bw = \lambda Uw + qq^*Bw.
\end{equation*}
From Algorithm~~\ref{alg:exact-DMD-sequential}, we see that an exact DMD mode~$\phi$ can be viewed as a projected DMD mode~$\hat\phi$ (calculated from~\eqref{eq:def-proj-DMD-mode}) plus a correction (the last term in~\eqref{eq:exact-mode-correction}) that lies in the nullspace of~$A$.



\section{Applications}
\label{sec:applications}

In this section we discuss the practical implications of a DMD framework based on Definition~\ref{def:exact-DMD}.
Specifically, we extend DMD to nonsequential datasets using the fact that Definition~\ref{def:exact-DMD} places no constraints on the structure of the data matrices $X$ and $Y$.
This allows for novel temporal sampling strategies that we show to increase computational efficiency and mitigate the effects of noise, respectively.
We also present canonical examples that demonstrate the potential benefits and pitfalls of applying DMD to rank-deficient data.
Such computations are not discussed in the existing DMD literature, but can be understood in our linear algebra-based framework.

\subsection{Nonsequential sampling}
\label{ssec:nonseq-sampling}

In Section~\ref{ssec:new-def}, there were no assumptions made on the data contained in $X$ and $Y$.
However, DMD is typically applied to data that come from a dynamical system, for instance one whose evolution is given by
\begin{equation*}
    z \mapsto f(z),
\end{equation*}
with $z \in \mathbb{R}^n$.
Often, the data consist of direct measurements of the state $z$.
More generally, one could measure some function of the state $h(z)$ (see Section~\ref{ssec:Koopman}).

Consider a set of vectors $\{z_1, \ldots, z_m\}$.
These vectors need not comprise a sequential time series; they do not even need to be sampled from the same dynamical trajectory.
There is no constraint on how the vectors $z_k$ are sampled from the state space.
We pair each vector $z_k$ with its image under the dynamics $f(z_k)$.
This yields a set of data pairs $\{(z_1, f(z_1)), \ldots, (z_m, f(z_m))\}$.
Arranging these vectors as in~\eqref{eq:def-XY-general} yields matrices
\begin{equation}
    \label{eq:def-XY-nonseq-sampling}
    X \triangleq \begin{bmatrix} z_1 & \cdots & z_m \end{bmatrix},
    \qquad
    Y = \begin{bmatrix} f(z_1) & \cdots & f(z_m) \end{bmatrix},
\end{equation}
with which we can perform either projected or exact DMD.
The case of sequential time series, for which $z_{k+1} = f(z_k)$, is simply a special case.

\begin{remark}
    We observe that only the pairwise correspondence of the columns of $X$ and $Y$ is important, and not the overall ordering.
    That is, permuting the order of the columns of $X$ and~$Y$ has no effect on the matrix $A = Y X^\pinv$, or on the subsequent computation of DMD modes and eigenvalues, so long as the same permutation is applied to the columns of both~$X$ and~$Y$.
    This is true even for data taken from a sequential time series (i.e., $X$ and $Y$ as given in~\eqref{eq:def-XY-sequential}).
\end{remark}

Recall from Definition~\ref{def:exact-DMD} and Remark~\ref{rmk:best-fit-operator} that DMD can be interpreted as an analysis of the best-fit linear operator relating $X$ and $Y$.
This operator relates the columns of $X$ to those of $Y$ in a pairwise fashion.
For $X$ and $Y$ as in~\eqref{eq:def-XY-nonseq-sampling}, the columns of $X$ are mapped to those of $Y$ by the (generally) nonlinear map $f$.
Regardless, DMD provides eigenvalues and eigenvectors of the linear map $A = Y X^\pinv$.
Then we can interpret DMD as providing an analysis of the best-fit linear approximation to $f$.
If $X$ and $Y$ are linearly consistent, we can also interpret DMD using the formalism of Koopman operator theory (see Section~\ref{ssec:Koopman}).

The use of~\eqref{eq:def-XY-nonseq-sampling} in place of~\eqref{eq:def-XY-sequential} certainly offers greater flexibility in the sampling strategies that can be employed for DMD (see Sections~\ref{sssec:nonuniform-sampling} and~\ref{sssec:comb-mult-traj}).
However, it is important to note that for sequential time series, there exist memory-efficient variants of Algorithm~\ref{alg:proj-DMD}~\cite{tuJCP12,belsonACMTMS13}.
These improved algorithms take advantage of the overlap in the columns of $X$ and~$Y$ (when defined as in~\eqref{eq:def-XY-sequential}) to avoid redundant computations; the same strategies can be applied to Algorithms~\ref{alg:exact-DMD},~\ref{alg:big-svd}, and~\ref{alg:exact-DMD-sequential}.
This is not possible for the more general definitions of $X$ and~$Y$ given by~\eqref{eq:def-XY-general} and~\eqref{eq:def-XY-nonseq-sampling}.

\subsection{Examples}
\label{ssec:examples}

The following section presents examples demonstrating the utility of a DMD theory based on Definition~\ref{def:exact-DMD}.
The first two examples consider DMD computations involving rank-deficient datasets, which are not treated in the existing DMD literature.
We show that in some cases, DMD can still provide meaningful information about the underlying dynamical system, but in others, the results can be misleading.
The second two examples use the generalized approach described in Section~\ref{ssec:nonseq-sampling} to perform DMD analysis using nonsequential datasets.
First, we use nonuniform sampling to dramatically increase the efficiency of DMD computations.
Then, we concatenate time series taken from multiple runs of an experiment, reducing the effects of noise.

\subsubsection{Stochastic dynamics}

Consider a system with stochastic dynamics
\begin{equation}
    z_{k+1} = \lambda z_k + n_k,
    \label{eq:noisy_1d_sys}
\end{equation}
where each $z_k \in \mathbb{R}$.
We choose a decay rate $\lambda = 0.5$ and let $n_k$ be white noise with variance $\sigma^2 = 10$.
(This system was first used as a test of DMD in~\cite{wynnJFM12}.)
Figure~\ref{fig:noisy_1d_sys}~(left)  shows a typical trajectory for an initial condition $z_0 = 0$.
If we apply DMD to this trajectory, we estimate a decay rate $\tilde{\lambda} = 0.55$.
This is despite the fact that the nominal (noiseless) trajectory is simply given by $z_k = 0$; a global, linear analysis of the trajectory shown in Figure~\ref{fig:noisy_1d_sys}~(left) would identify a stationary process ($\tilde{\lambda} = 0$).

\begin{figure}[b]
    \centering
    \includegraphics{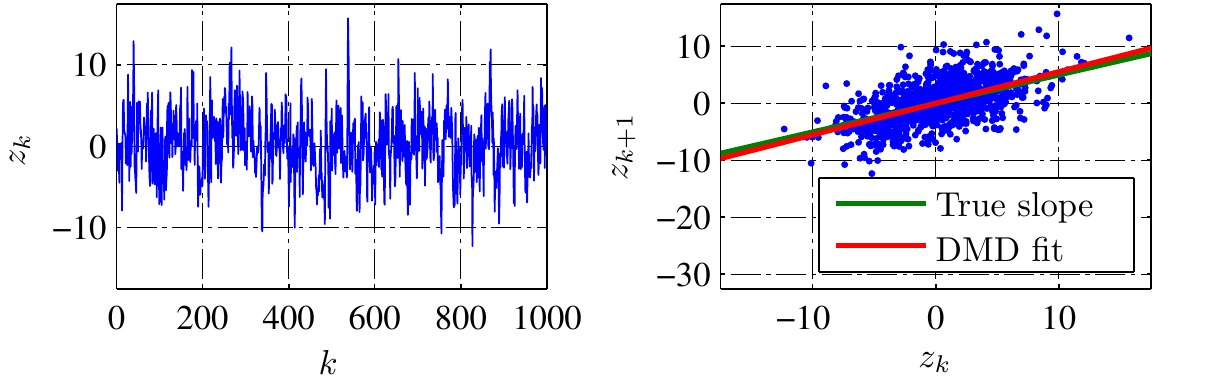}
    \caption{(Left) Typical trajectory of the noisy one-dimensional system governed by~\eqref{eq:noisy_1d_sys}.
    (Right) Scatter plot showing the correlation of $z_{k+1}$ and $z_k$.
    DMD is able to identify the relationship between future and past values of $z$ even though the dataset is rank-deficient.}
    \label{fig:noisy_1d_sys}
\end{figure}

Because the existing DMD literature focuses on high-dimensional systems, existing DMD theory deals primarily with time series whose elements are linearly independent.
As such, it cannot be applied to explain the ability of DMD to accurately estimate the dynamics underlying this noisy data (a rank-one time series).
Recalling Definition~\ref{def:exact-DMD}, we can interpret DMD in terms of a linear operator that relates the columns of a data matrix $X$ to those of $Y$, in column-wise pairs.
Figure~\ref{fig:noisy_1d_sys}~(right) shows the time series from Figure~\ref{fig:noisy_1d_sys}~(left) plotted in this pairwise fashion.
We see that though the data are noisy, there is clear evidence of a linear relationship between $z_k$ and $z_{k+1}$.
For rank-deficient data, DMD approximates the dynamics relating $X$ and $Y$ through a least-squares fit, and so it is no surprise that we can accurately estimate $\lambda$ from this time series.

\subsubsection{Standing waves}
\label{sssec:standing-wave}

Because each DMD mode has a corresponding DMD eigenvalue (and thus a corresponding growth rate and frequency), DMD is often used to analyze oscillatory behavior, whether the underlying dynamics are linear or nonlinear.
Consider data describing a standing wave:
\begin{equation}
    \label{eq:standing_wave}
    z_k = \cos(k\theta) q,\qquad k=0,\ldots,m,
\end{equation}
where $q$ is a fixed vector in $\mathbb{R}^n$.
For instance, such data can arise from the linear system
\begin{equation}
    \label{eq:standing-wave-sys}
    \begin{aligned}
        u_{k+1} &= (\cos\theta)u_k - (\sin\theta)v_k\\
        v_{k+1} &= (\sin\theta)u_k + (\cos\theta)v_k
    \end{aligned}
    \qquad (u_0,v_0) = (q,0),
\end{equation}
where $(u_k, v_k) \in \mathbb{R}^{2n}$.
If we measure only the state $u_k$, then we observe the standing wave~\eqref{eq:standing_wave}.
Such behavior can also arise in nonlinear systems, for instance by measuring only one component of a multi-dimensional limit cycle.

Suppose we compute DMD modes and eigenvalues from data satisfying~\eqref{eq:standing_wave}.
By construction, the columns of the data matrix $X$ will be spanned by the single vector~$q$.
As such, the SVD of $X$ will generate a matrix $U$ with a single column, and the matrix $\tilde{A}$ will be $1 \times 1$.
Then there will be precisely one DMD eigenvalue~$\lambda$.
Since $z$ is real-valued, then so is $\lambda$, meaning it captures only exponential growth/decay, and no oscillations.
This is despite the fact that the original data are known to oscillate with a fixed frequency.

What has gone wrong?
It turns out that, in this example, the data are not linearly consistent (see~Definition~\ref{def:linear-consistency}).
To see this, let $x$ and~$y$ be vectors with components $x_k=\cos(k\theta)$ and $y_k=\cos((k+1)\theta)$ (for $k=0,\ldots,m-1$), so the data matrices become
\begin{equation*}
  X = q x^T,\qquad Y = q y^T.
\end{equation*}
Then $X$ and~$Y$ are not linearly consistent unless $\theta$ is a multiple of~$\pi$.

For instance, the vector $a=(-\cos\theta,1,0,\ldots,0)$ is in $\nullspace(X)$, since $x^Ta=0$.
However, $y^*a=-\cos^2\theta+\cos 2\theta = \sin^2\theta$, so $a\notin\nullspace(Y)$ unless $\theta=j\pi$.
Also, note that if $\theta=\pi$, then the columns of~$X$ simply alternate sign, and in this case DMD yields the (correct) eigenvalue~$-1$.
As such, even though the data in this example arise from the linear system~\eqref{eq:standing-wave-sys}, there is no $A$ such that $Y=AX$ (by Theorem~\ref{thm:AX-eq-Y}), and DMD fails to capture the correct dynamics.
This example underscores the importance of linear consistency.

We note that in practice, we observe the same deficiency when the data do not satisfy~\eqref{eq:standing_wave} exactly, so long as the dynamics are dominated by such behavior (a standing wave).
Thus the presence of random noise, which may increase the rank of the dataset, does not alleviate the problem.
This is not surprising, as the addition of random noise should not enlarge the subspace in which the oscillation occurs.
However, if we append the measurement with a time-shifted value, i.e., performing DMD on a sequence of vectors $\begin{bmatrix} z_k^T & z_{k+1}^T \end{bmatrix}^T$, then the data matrices $X$ and~$Y$ become linearly consistent, and we are able to identify the correct oscillation frequency.
(See Section~\ref{ssec:era} for an alternative motivation for this approach.)

\subsubsection{Nonuniform sampling}
\label{sssec:nonuniform-sampling}

Systems with a wide range of time scales can be challenging to analyze.
If data are collected too slowly, dynamics on the fastest time scales will not be captured.
On the other hand, uniform sampling at a high frequency can yield an overabundance of data, which can prove challenging to deal with numerically.
Such a situation can be handled using the following sampling strategy:
\begin{equation}
    \label{eq:nonuniform_sampling}
    X \triangleq
    \begin{bmatrix}
        \vline & \vline & & \vline \\
        z_0 & z_P & \ldots & z_{(m-1)P} \\
        \vline & \vline & & \vline
    \end{bmatrix}, \qquad
    Y \triangleq
    \begin{bmatrix}
        \vline & \vline & & \vline \\
        z_1 & z_{P + 1} & \ldots & z_{(m-1)P + 1} \\
        \vline & \vline & & \vline
    \end{bmatrix},
\end{equation}
where we again assume dynamics of the form $z_{k+1} = f(z_k)$.
The columns of $X$ and $Y$ are separated by a single iteration of $f$, capturing its fastest dynamics.
However, the tuning parameter $P$ allows for a separation of time scales between the flow map iteration and the rate of data collection.

We demonstrate this strategy using a flow control example.
Consider the flow past a two-dimensional cylinder, which is governed by the incompressible Navier--Stokes equations.
We simulate the dynamics using the fast immersed boundary projection method detailed in~\cite{tairaJCP07,coloniusCMAME08}.
The (non-dimensionalized) equations of motion are
\begin{gather*}
    \frac{\partial \vec{u}}{\partial t} + (\vec{u} \cdot \nabla) \vec{u} = - \nabla p + \frac{1}{\text{Re}} \nabla^2 \vec{u} + \int_{\partial \mathcal{B}} \vec{f}(\vec{x}) \delta(\vec{x} - \vec{\xi}) \ d\vec{\xi} \\
    \nabla \cdot \vec{u} = 0,
\end{gather*}
where $\vec{u}$ is the velocity, $p$ is the pressure, and $\vec{x}$ is the spatial coordinate.
The Reynolds number $\text{Re} \triangleq U_\infty D / \nu$ is a nondimensional paramter defined by the freestream velocity $U_\infty$, the cylinder diameter $D$, and the kinematic viscosity $\nu$.
$\partial \mathcal{B}$ is the union of the boundaries of any bodies in the flow.
$\vec{f}$ is a boundary force that can be thought of a Lagrange multiplier used to enforce the no-slip boundary condition.
$\delta$ is the Dirac delta function.

We consider a cylinder with diameter $D = 1$ placed at the origin, in a flow with freestream velocity $U_\infty = 1$ and Reynolds number $\text{Re} = 100$.
Convergence tests show that our simulations are accurate for an inner-most domain with $(x, y) \in [15, 15] \times [-5, 5]$ and a $1500 \times 500$ mesh of uniformly-spaced points.
At this Reynolds number, the flow is globally unstable.
As a step towards computing a reduced-order model using BPOD~\cite{rowleyIJBC05}, we restrict the linearized dynamics to their stable subspace.
The system is actuated using a disk of vertical velocity downstream of the cylinder and sensed using localized measurements of vertical velocity placed along the flow centerline.
This setup is based on flow control benchmark proposed in~\cite{noackACC04} and is illustrated in Figure~\ref{fig:2d_cyl_sens_act_and_imp_resp}~(left).

\begin{figure}[t]
    \centering
    \includegraphics{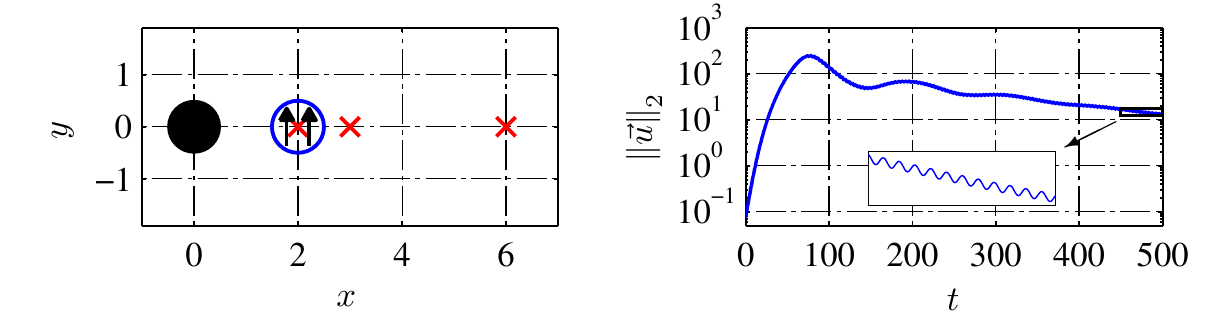}
    \caption{(Left) Schematic showing the placement of sensors ({\color{red}$\times$}) and actuators ({\color{blue}$\circ$}) used to control the flow past a two-dimensional cylinder.
    (Right) Kinetic energy of the corresponding impulse response (restricted to the stable subspace).
    After an initial period of non-normal growth, oscillations with both short and long time scales are observed.}
    \label{fig:2d_cyl_sens_act_and_imp_resp}
\end{figure}

The impulse response of this system is shown in Figure~\ref{fig:2d_cyl_sens_act_and_imp_resp}~(right).
We see that from $t = 200$ to $t = 500$, the dynamics exhibit both a fast and slow oscillation.
Suppose we want to identify the underlying frequencies and corresponding modes using DMD.
For instance, this would allow us to implement the more efficient and more accurate analytic tail variant of BPOD~\cite{tuJCP12}.
In order to capture the fast frequency, we must sample the system every 50 timesteps, with each timestep corresponding to $\Delta t = 0.02$.
(This is in order to satisfy the Nyquist--Shannon sampling criterion.)
As such, we let $z_k = z(50 k \Delta t)$.

Figure~\ref{fig:2d_cyl_dmd_spectrum} and Table~\ref{tbl:2d_cyl_dmd_eig_vals} compare the DMD eigenvalues computed using uniform sampling and nonuniform sampling.
(Referring back to \eqref{eq:nonuniform_sampling}, the former corresponds to $P = 1$ and the latter to $P = 10$.)
We see that the dominant eigenvalues agree, with less than 10\% error in all cases; we use the DMD eigenvalues computed with uniform sampling as truth values.
However, the larger errors occur for modes with norms on the order of $10^{-5}$, two orders of magnitude smaller than those of the dominant two DMD modes.
As such, these modes have negligible contribution to the evolution of the impulse response, and the error in the corresponding eigenvalues is not significant.
The dominant DMD modes show similar agreement, as seen in Figure~\ref{fig:2d_cyl_dmd_modes}.
This agreement is achieved despite using 90\% less data in the nonuniform sampling case, which corresponds to an 85.8\% reduction in computation time.

\begin{figure}[t]
    \centering
    \includegraphics{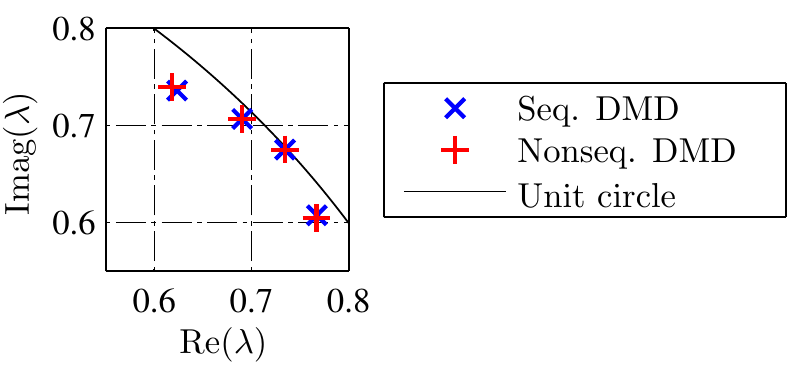}
    \caption{DMD estimates of the eigenvalues governing the decay of the impulse response shown in Figure~\ref{fig:2d_cyl_sens_act_and_imp_resp}~(right).
    The slowest decaying eigenvalues are captured well with both uniform sampling (sequential DMD) and nonuniform sampling (nonsequential DMD).}
    \label{fig:2d_cyl_dmd_spectrum}
\end{figure}

\subsubsection{Combining multiple trajectories}
\label{sssec:comb-mult-traj}

DMD is often applied to experimental data, which are typically noisy.
While filtering or phase averaging can be done to eliminate noise prior to DMD analysis, this is not always desirable, as it may remove features of the true dynamics.
In POD analysis, the effects of noise can be averaged out by combining multiple trajectories in a single POD computation.
We can take the same approach in DMD analysis using~\eqref{eq:def-XY-nonseq-sampling}.

\begin{table}[t]
    \centering
    \caption{Comparison of DMD eigenvalues$^*$}
    \begin{threeparttable}
        \begin{tabular}{ccccccc}
            \toprule
            \multicolumn{3}{c}{Frequency} & & \multicolumn{3}{c}{Growth rate} \\
            \cmidrule{1-3} \cmidrule{5-7}
            Seq. DMD & Nonseq. DMD & Error & & Seq. DMD &  Nonseq. DMD  & Error \\
            0.118 & 0.118 & 0.00\% & & 0.998 & 0.998 & 0.00\% \\
            0.127 & 0.127 & 0.01\% & & 0.988 & 0.988 & 0.00\% \\
            0.107 & 0.106 & 3.40\% & & 0.979 & 0.977 & 2.10\% \\
            0.138 & 0.139 & 7.50\% & & 0.964 & 0.964 & 0.05\% \\
            \bottomrule
        \end{tabular}
        \begin{tablenotes}
            \item[*]{\small{Row order corresponds to decreasing mode norm.}}
        \end{tablenotes}
    \end{threeparttable}
    \label{tbl:2d_cyl_dmd_eig_vals}
\end{table}

\begin{figure}[t]
    \centering
    \includegraphics{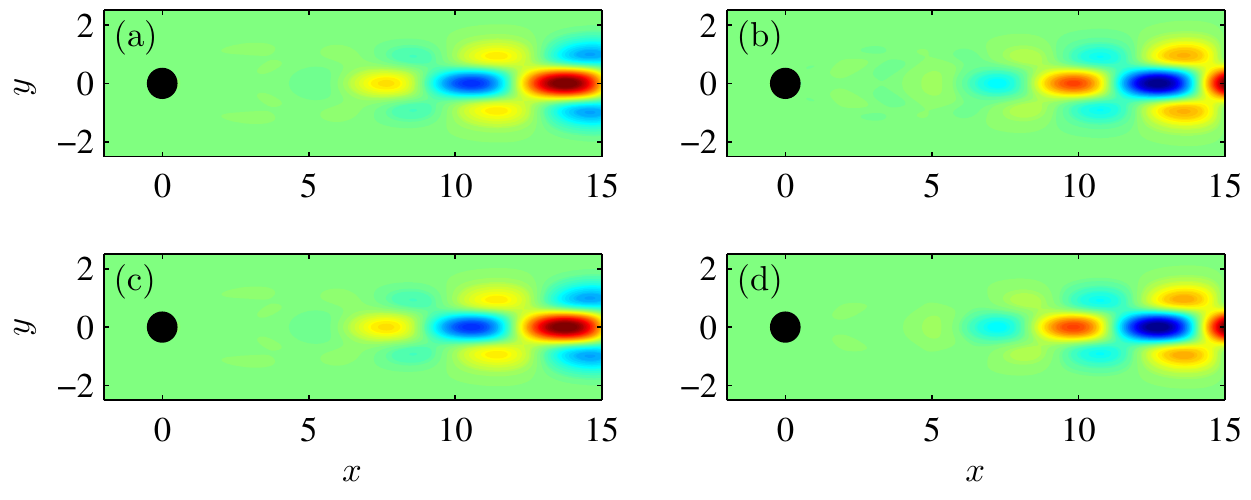}
    \caption{Comparison of dominant DMD modes computed from the impulse response shown in Figure~\ref{fig:2d_cyl_sens_act_and_imp_resp}~(right), illustrated using contours of vorticity.
    For each of the dominant frequencies, modes computed using nonuniform sampling (nonsequential DMD; bottom row) match those computed using uniform sampling (sequential DMD; top row).
    (For brevity, only the real part of each mode is shown; similar agreement is observed in the imaginary parts.)
    (a) $f = 0.118$, uniform sampling; (b) $f = 0.127$, uniform sampling; (c) $f = 0.118$, nonuniform sampling; (d) $f = 0.127$, nonuniform sampling.}
    \label{fig:2d_cyl_dmd_modes}
\end{figure}

Consider multiple dynamic trajectories, indexed by $j$: $\{z_k^j\}_{k=0}^{m_j}$.
These could be multiple runs of an experiment, or particular slices of a single, long trajectory.
(The latter might be useful in trying to isolate the dynamics of a recurring dynamic event.)
Suppose there are a total of $J$ trajectories.
DMD can be applied to the entire ensemble of trajectories by defining
\begin{equation}
    \label{eq:multiple_trajectories}
    \begin{array}{ccccccccccccc}
        X \hspace{-2.5mm} &
        \triangleq \hspace{-2.5mm} &
        \left[\begin{matrix} \vline \\ z_0^0 \\ \vline \end{matrix}\right. &
        \begin{matrix} \phantom{\vline} \\ \ldots \\ \phantom{\vline} \end{matrix} &
        \begin{matrix} \vline \\ z_{m_0-1}^0 \\ \vline \end{matrix} &
        \begin{matrix} \vline \\ z_0^1 \\ \vline \end{matrix} &
        \begin{matrix} \phantom{\vline} \\ \ldots \\ \phantom{\vline} \end{matrix} &
        \begin{matrix} \vline \\ z_{m_1-1}^1 \\ \vline \end{matrix} &
        \begin{matrix} \phantom{\vline} \\ \ldots \\ \phantom{\vline} \end{matrix} &
        \begin{matrix} \vline \\ z_0^J \\ \vline \end{matrix} &
        \begin{matrix} \phantom{\vline} \\ \ldots \\ \phantom{\vline} \end{matrix} &
        \begin{matrix} \vline \\ z_{m_J-1}^J \\ \vline \end{matrix} &
        \hspace{-4.1mm}\left.\phantom{\begin{matrix} \phantom{\vline} \\ \\ \phantom{\vline} \end{matrix}}\right],
        \\
        \\
        Y \hspace{-2.5mm}  &
        \triangleq \hspace{-2.5mm}  &
        \left[\begin{matrix} \vline \\ z_1^0 \\ \vline \end{matrix}\right. &
        \begin{matrix} \phantom{\vline} \\ \ldots \\ \phantom{\vline} \end{matrix} &
        \begin{matrix} \vline \\ z_{m_0}^0 \\ \vline \end{matrix} &
        \begin{matrix} \vline \\ z_1^1 \\ \vline \end{matrix} &
        \begin{matrix} \phantom{\vline} \\ \ldots \\ \phantom{\vline} \end{matrix} &
        \begin{matrix} \vline \\ z_{m_1}^1 \\ \vline \end{matrix} &
        \begin{matrix} \phantom{\vline} \\ \ldots \\ \phantom{\vline} \end{matrix} &
        \begin{matrix} \vline \\ z_1^J \\ \vline \end{matrix} &
        \begin{matrix} \phantom{\vline} \\ \ldots \\ \phantom{\vline} \end{matrix} &
        \begin{matrix} \vline \\ z_{m_J}^J \\ \vline \end{matrix} &
        \hspace{-4.1mm}\left.\phantom{\begin{matrix} \phantom{\vline} \\ \\ \phantom{\vline} \end{matrix}}\right].
    \end{array}
\end{equation}

We demonstrate this approach using experimental data from a bluff-body wake experiment.
A finite-thickness flat plate with an elliptical leading edge is placed in a uniform oncoming flow.
Figure~\ref{fig:bluff_body_wake_schematic} shows a schematic of the experimental setup.
We capture snapshots of the velocity field in the wake behind the body using a time-resolved particle image velocimetry (PIV) system.
(For details on the PIV data acquisition, see~\cite{tuEF13}.)
Multiple experimental runs are conducted, with approximately 1,400 velocity fields captured in each run.
This corresponds to the maximum amount of data that can be collected per run.

\begin{figure}[t]
    \centering
    \includegraphics[width=5in]{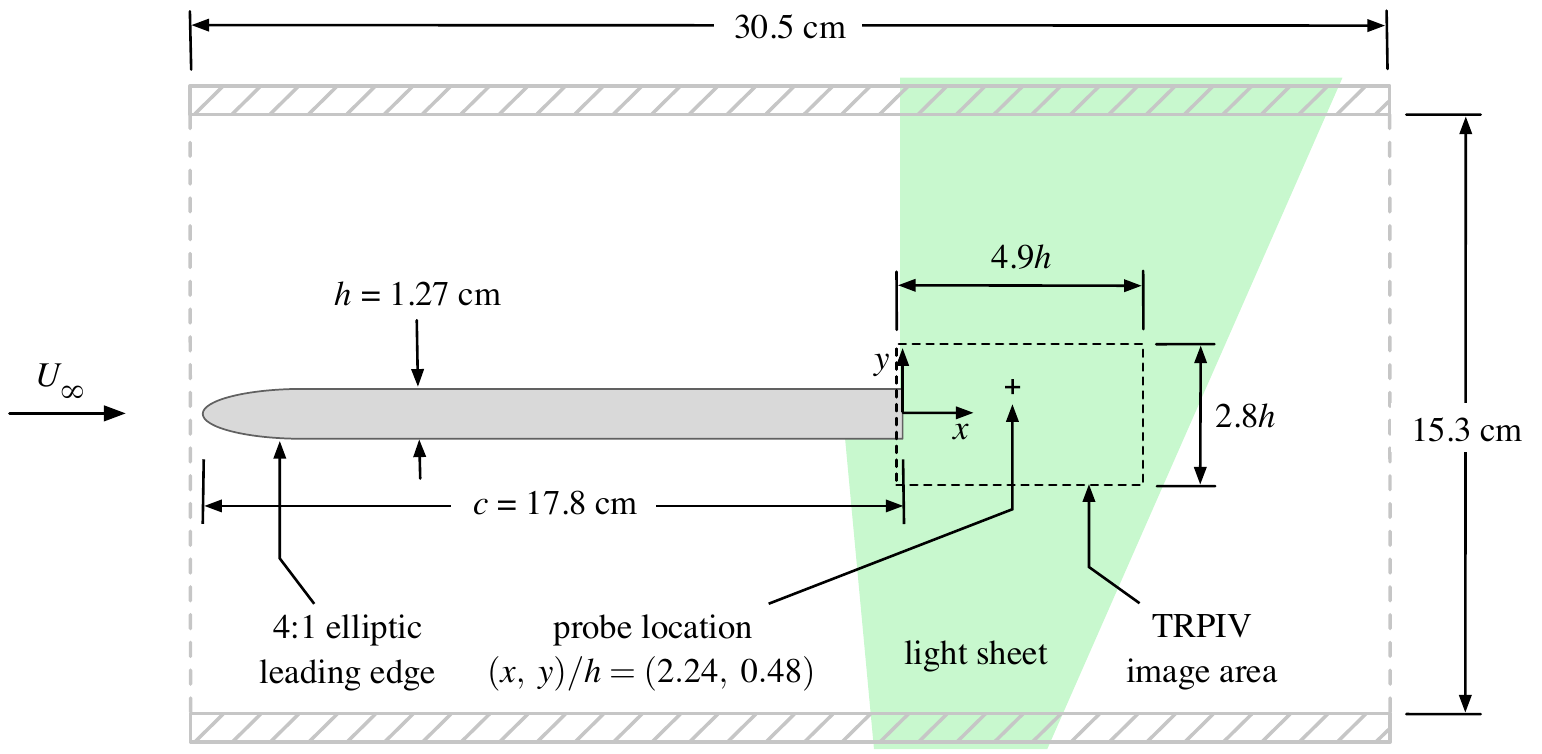}
    \caption{Schematic of setup for bluff-body wake experiment.}
    \label{fig:bluff_body_wake_schematic}
\end{figure}

Due to the high Reynolds number ($\text{Re} = 50,000$), the flow is turbulent.
As such, though we observe a standard von K\'{a}rm\'{a}n vortex street (see Figure~\ref{fig:bluff_body_wake_PIV}), the familiar vortical structures are contaminated by turbulent fluctuations.
Figure~\ref{fig:bluff_body_wake_dmd_spectra}~(left) shows a DMD spectrum computed using PIV data from a single experimental run.\footnote{Instead of simply plotting the mode norms against their corresponding frequencies, as is generally done, we first scale the mode norms by $\lambda^m$.
This reduces the height of spectral peaks corresponding to modes with large norm but quickly decaying eigenvalues.
For dynamics known to lie on an attractor, such peaks can be misleading; they do not contribute to the long-time evolution of the system.}
The spectrum is characterized by a harmonic set of peaks, with the dominant peak corresponding to the wake shedding frequency.
The corresponding modes are shown in Figure~\ref{fig:bluff_body_wake_dmd_modes_compare}~(a--c).
We see that the first pair of modes (see Figure~\ref{fig:bluff_body_wake_dmd_modes_compare}~(a)) exhibits approximate top-bottom symmetry, with respect to the centerline of the body.
The second pair of modes (see Figure~\ref{fig:bluff_body_wake_dmd_modes_compare}~(b)) shows something close to top-bottom antisymmetry, though variations in the vorticity contours make this antisymmetry inexact.
The third pair of modes (see Figure~\ref{fig:bluff_body_wake_dmd_modes_compare}~(c)) also shows approximate top-bottom symmetry, with structures that are roughly spatial harmonics of those seen in the first mode pair.

These modal features are to be expected, based on two-dimensional computations of a similar flow configuration~\cite{tuAIAA11}.
However, when computed from noise-free simulation data, the symmetry/antisymmetry of the modes is more exact.
Figures~\ref{fig:bluff_body_wake_dmd_spectra}~(right) and~\ref{fig:bluff_body_wake_dmd_modes_compare}~(d--g) show that when five experimental runs are used, the experimental DMD results improve, more closely matching computational results.
In the DMD spectrum (see Figure~\ref{fig:bluff_body_wake_dmd_spectra}~(right)), we again observe harmonic peaks, with a fundamental frequency corresponding to the shedding frequency.
The peaks are more isolated than those in Figure~\ref{fig:bluff_body_wake_dmd_spectra}~(left); in fact, we observe a fourth frequency peak, which is not observed in the single-run computation.
The modal structures, shown in Figure~\ref{fig:bluff_body_wake_dmd_modes_compare}~(d--g), display more obvious symmetry and antisymmetry, respectively.
The structures are also smoother and more elliptical.

\begin{figure}
    \centering
    \includegraphics{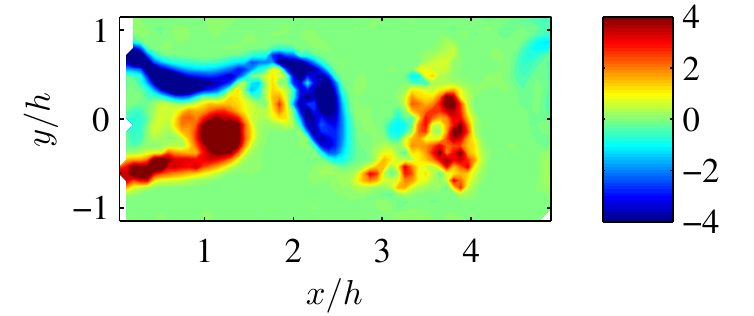}
    \caption{Typical vorticity field from the bluff-body wake experiment depicted in Figure~\ref{fig:bluff_body_wake_schematic}.
    A clear von K\'{a}rm\'{a}n vortex street is observed, though the flow field is contaminated by turbulent fluctuations.}
    \label{fig:bluff_body_wake_PIV}
\end{figure}

\begin{figure}
    \centering
    \includegraphics{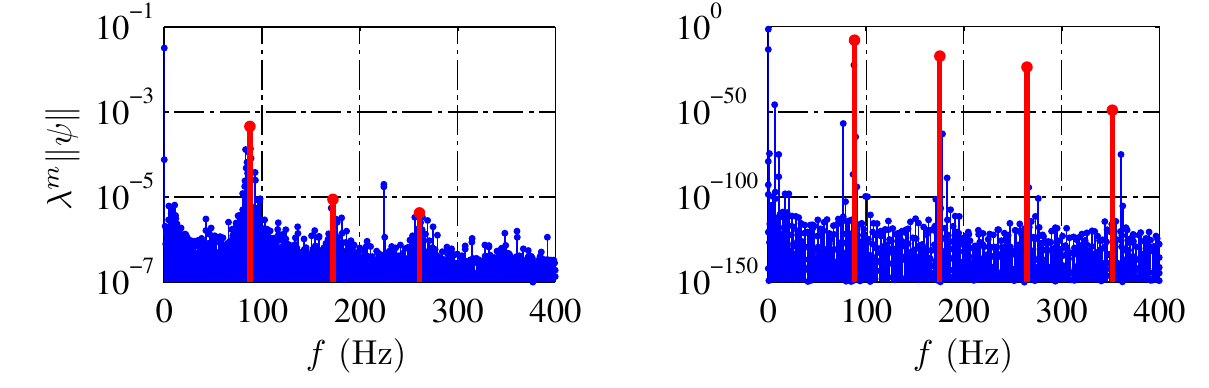}
    \caption{Comparison of DMD spectra computed using a single experimental run (left) and five experimental runs (right).
    When multiple runs are used, the spectral peaks are more isolated and occur at almost exactly harmonic frequencies.
    Furthermore, a fourth harmonic peak is identified; this peak is obscured in the single-run DMD computation.
    (Spectral peaks corresponding to modes depicted in Figure~\ref{fig:bluff_body_wake_dmd_modes_compare} are shown in red.)}
    \label{fig:bluff_body_wake_dmd_spectra}
\end{figure}

\begin{figure}
    \centering
    \includegraphics{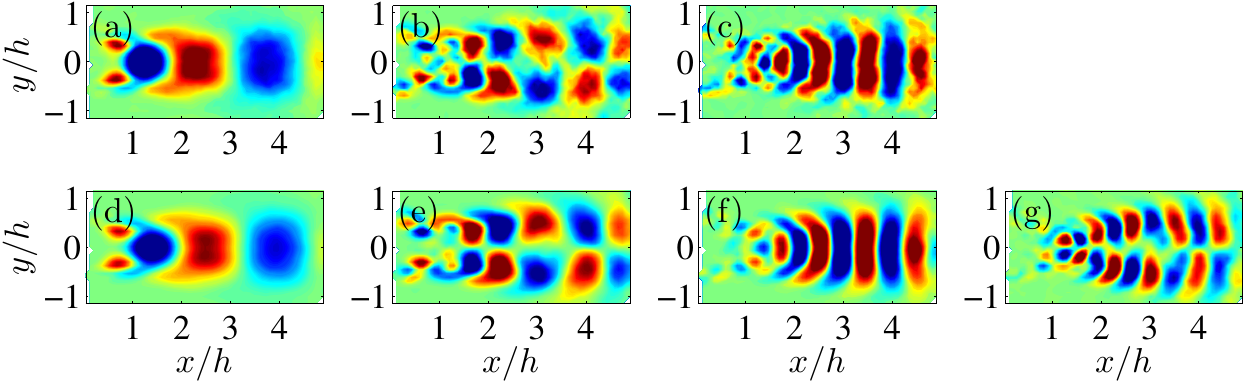}
    \caption{Representative DMD modes, illustrated using contours of vorticity.
    (For brevity, only the real part of each mode is shown.)
    The modes computed using multiple runs (bottom row) have more exact symmetry/antisymmetry and smoother contours.
    Furthermore, with multiple runs four dominant mode pairs are identified; the fourth spectral peak is obscured in the single-run computation (see Figure~\ref{fig:bluff_body_wake_dmd_spectra}).
    (a) $f = 87.75$~Hz, single run; (b) $f = 172.6$~Hz, single run; (c) $f = 261.2$~Hz, single run; (d) $f = 88.39$~Hz, five runs; (e) $f = 175.6$~Hz, five runs; (f) $f = 264.8$~Hz, five runs; (g) $f = 351.8$~Hz, five runs.}
    \label{fig:bluff_body_wake_dmd_modes_compare}
\end{figure}


\section{Connections to other methods}
\label{sec:connections}

In this section we discuss how DMD relates to other methods.
First, we show that our definition of DMD preserves, and even strengthens, the connections between DMD and Koopman operator theory.
Without these connections, the use of DMD to analyze nonlinear dynamics appears dubious, since there seems to be an underlying assumption of (approximately) linear dynamics (see Section~\ref{ssec:nonseq-sampling}), as in~\eqref{eq:lin-dyn}.
One might well question whether such an approximation would characterize a nonlinear system in a meaningful way.
However, so long as DMD can be interpreted as an approximation to Koopman spectral analysis, there is a firm theoretical foundation for applying DMD in analyzing nonlinear dynamics.

Second, we explore the links between DMD and the eigensystem realization algorithm (ERA).
The close relationship between the two methods provides motivation for the use of strategies from the ERA in DMD computations where rank is a problem.
Finally, we show that under certain assumptions, DMD is equivalent to linear inverse modeling (LIM), a method developed in the climate science community decades ago.
The link between the two methods suggests that practitioners of DMD may benefit from an awareness of the LIM literature, and vice versa.

\subsection{Koopman spectral analysis}
\label{ssec:Koopman}

We briefly introduce the concepts of Koopman operator theory below and discuss how they relate to the theory outlined in Section~\ref{sec:theory}.
The connections between Koopman operator theory and projected DMD were first explored in~\cite{rowleyJFM09}, but only in the context of sequential time series.
Here, we extend those results to more general datasets, doing so using exact DMD (defined in Section~\ref{ssec:new-def}).

Consider a discrete-time dynamical system
\begin{equation}
    \label{eq:nonlin-dyn}
    z \mapsto f(z),
\end{equation}
where $z$ is an element of a finite-dimensional manifold~$M$.
The Koopman operator $\mathcal{K}$ acts on scalar functions $g:M \rightarrow \mathbb{R}$~or~$\mathbb{C}$, mapping $g$ to a new function $\mathcal{K}g$ given by
\begin{equation}
    \label{eq:def_Koopman}
    \mathcal{K}g(z) \triangleq g\big(f(z)\big).
\end{equation}
Thus the Koopman operator is simply a composition or pull-back operator.
We observe that $\mathcal{K}$ acts linearly on functions $g$, even though the dynamics defined by~$f$ may be nonlinear.

As such, suppose $\mathcal{K}$ has eigenvalues $\lambda_j$ and eigenfunctions
$\theta_j$, which satisfy
\begin{equation}
    \label{eq:def_Koopman_eigfunc}
    \mathcal{K} \theta_j(z) = \lambda_j \theta_j(z), \quad j = 0,1, \ldots
\end{equation}
An \emph{observable} is simply a function on the state space~$M$.
Consider a vector-valued observable $h:M \rightarrow \mathbb{R}^n$~or~$\mathbb{C}^n$, and expand~$h$ in terms of these eigenfunctions as
\begin{subequations}
    \label{eq:Koopman_dyn}
\begin{equation}
    \label{eq:Koopman_dyn_a}
    h(z) = \sum_{j=0}^\infty \theta_j(z) \tilde{\phi}_j,
\end{equation}
where $\tilde{\phi}_j \in \mathbb{R}^n$~or~$\mathbb{C}^n$.
(We assume each component of $h$ lies in the span of the eigenfunctions.)
We refer to the vectors $\tilde{\phi}_j$ as \emph{Koopman modes} (after~\cite{rowleyJFM09}).

Applying the Koopman operator, we find that
\begin{equation}
    \label{eq:Koopman_dyn_b}
    h\big(f(z)\big) = \sum_{j=0}^\infty \lambda_j \theta_j(z) \tilde{\phi}_j.
\end{equation}
\end{subequations}
For a sequential time series, we can repeatedly apply the Koopman operator (see~\eqref{eq:def_Koopman} and~\eqref{eq:def_Koopman_eigfunc}) to find
\begin{equation}
    \label{eq:Koopman_dyn_seq}
    h(f^k(z)) = \sum_{j=0}^\infty \lambda_j^k \theta_j(z) \tilde{\phi}_j.
\end{equation}
A set of modes $\{\tilde{\phi}_j\}$ and eigenvalues $\{\lambda_j\}$ must satisfy~\eqref{eq:Koopman_dyn} in order to be considered Koopman modes and eigenvalues.

Now consider a set of arbitrary initial states $\{z_0,z_1,\ldots,z_{m-1}\}$ (not necessarily a trajectory of the dynamical system~\eqref{eq:nonlin-dyn}), and let
\begin{equation*}
  x_k = h(z_k),\qquad y_k=h(f(z_k)).
\end{equation*}
As before, construct data matrices $X$ and~$Y$, whose columns are $x_k$ and $y_k$.
This is similar to~\eqref{eq:def-XY-nonseq-sampling}, except that here we measure an observable, rather than the state itself.
As long as the data matrices $X$ and~$Y$ are linearly consistent (see Definition~\ref{def:linear-consistency}), we can determine a relationship between DMD modes and Koopman modes, as follows.

Suppose we compute DMD modes from $X$ and $Y$, giving us
eigenvectors and eigenvalues that satisfy
\begin{equation*}
    A \phi_j = \lambda_j \phi_j,
\end{equation*}
where $A=YX^\pinv$.  If the matrix~$A$ has a full set of eigenvectors, then we can expand each column
$x_k$ of~$X$ as
\begin{subequations}
    \label{eq:dmd_reconstruction}
\begin{equation}
    h(z_k) = x_k = \sum_{j=0}^{n-1} c_{jk} \phi_j,
\end{equation}
for some constants $c_{jk}$.
For linearly consistent data, we have $A x_k = y_k$, by Theorem~\ref{thm:AX-eq-Y},
so
\begin{align}
    h\big(f(z_k)\big) &= y_k = A x_k \notag \\
    &= \sum_{j=0}^{n-1} c_{jk} A \phi_j \notag \\
    &= \sum_{j=0}^{n-1} \lambda_j c_{jk} \phi_j.
\end{align}
\end{subequations}

Comparing to~\eqref{eq:Koopman_dyn}, we see that in the case of linearly consistent data matrices (and diagonalizable~$A$), the DMD modes correspond to Koopman modes, and the DMD eigenvalues to Koopman eigenvalues, where the constants $c_{jk}$ are given by $c_{jk}=\theta_j(z_k)$.
(We note, however, that we typically do not know the states $z_k$, nor do we know the Koopman eigenfunctions $\theta_j$.)
If the data arise from a sequential time series, we can scale the modes to subsume the constants $c_{jk}$, following~\cite{chenJNLS12}.
(For more details see Appendix~\ref{sec:mode-scalings}.)

The relationship between Koopman modes and DMD modes is similar to that established in~\cite{rowleyJFM09} (and discussed further in~\cite{chenJNLS12}), but the present result differs in some important respects.
First, the result in~\cite{rowleyJFM09} uses projected DMD modes (see Section~\ref{sec:standard-def}) and requires a sequential time series; here, we use exact DMD modes (see Section~\ref{ssec:new-def}) and do not require a sequential time series.
Furthermore,~\cite{rowleyJFM09} assumes the vectors~$x_k$ are linearly independent; here, we impose the weaker requirement of linear consistency.

We note that when $n$ is large, it may be impractical to compute all of the DMD modes.
For instance, when $n \gg m$, $A$ will have a large nullspace, and one might use Algorithm~\ref{alg:exact-DMD} to compute only those DMD modes with nonzero eigenvalues.
Suppose the rank of $A$ is $r$ and we write
\begin{subequations}
    \label{eq:partitioned_dmd_reconstruction}
    \begin{equation}
        x_k = \sum_{j=0}^{r-1} c_{jk} \phi_j + \sum_{j=r}^{n-1} c_{jk} \phi_j,
    \end{equation}
    where the first sum contains only DMD modes with nonzero eigenvalues.
    We still have
    \begin{equation}
        y_k = \sum_{j=0}^{r-1} \lambda_j c_{jk} \phi_j,
    \end{equation}
\end{subequations}
since all modes in the second sum have zero eigenvalues.
As such, those modes don't contribute to the dynamics, and we can neglect those vectors as an error term that gets projected out by the dynamics.
Contrast this with the DMD reconstruction of a sequential time series using projected DMD, where the residual instead appears at the end of the time series (see 3.13--3.14 in~\cite{rowleyJFM09}).

Though the Koopman analogy provides a firm mathematical foundation for applying DMD to data generated by nonlinear systems, it is limited by the fact that it relies on~\eqref{eq:dmd_reconstruction}.
The derivation of this equation requires making a number of assumptions, namely that the data are linearly consistent and that the matrix~$A$ has a full set of eigenvectors (e.g., this holds when the eigenvalues of~$A$ are distinct).
When these assumptions do not hold, there is no guarantee that DMD modes will closely approximate Koopman modes.
For instance, in some systems DMD modes and eigenvalues closely approximate those of the Koopman operator near an attractor, but not far from it~\cite{bagheriJFM13}.
DMD may also perform poorly when applied to dynamics whose Koopman spectral decomposition contains Jordan blocks~\cite{bagheriJFM13}.
In contrast, an understanding of DMD built on Definition~\ref{def:exact-DMD} holds even when these conditions break down.

\subsection{The eigensystem realization algorithm (ERA)}
\label{ssec:era}

The ERA is a control-theoretic method for system identification and model reduction~\cite{hoACCST65,juangJGCD85,maTCFD11}.
Applicable to linear systems, the ERA takes input-output data and from them computes a minimal realization of the underlying dynamics.
In this section, we show that while DMD and the ERA were developed in different contexts, they are closely related: the low-order linear operators central to each method are related by a similarity transform.
This connection suggests that strategies used in ERA computations could be leveraged for gain in DMD computations.
Specifically, it provides a motivation for appending the data matrices $X$ and $Y$ with time-shifted data to overcome rank problems (as suggested in Section~\ref{sssec:standing-wave}).

Consider the linear, time-invariant system
\begin{equation} \begin{aligned}
    x_{k+1} &= A x_k + B u_k  \\
    y_k &= C x_k + D u_k,
    \label{eq:era_sys}
\end{aligned} \end{equation}
where $x \in \mathbb{R}^n$, $u \in \mathbb{R}^p$, and $y \in \mathbb{R}^q$.
(The matrix $A$ defined here is not necessarily related to the one defined in Section~\ref{sec:theory}.)
We refer to $x$ as the state of the system, $u$ as the input, and $y$ as the output.

The goal of the ERA is to identify the dynamics of the system~\eqref{eq:era_sys} from a time history of $y$.
Specifically, in the ERA we sample outputs from the impulse response of the system.
We collect the sampled outputs in two sets
\begin{equation*} \begin{gathered}
    \mathcal{H} \triangleq \left\{CB, C A^P B, \ldots, C A^{(m-1)P} B \right\} \\
    \mathcal{H}' \triangleq \left\{CAB, C A^{P+1} B, \ldots, C A^{(m-1)P+1} B \right\},
\end{gathered} \end{equation*}
where we sample the impulse response every $P$ steps.
The elements of $\mathcal{H}$ and $\mathcal{H}'$ are commonly referred to as Markov parameters of the system~\eqref{eq:era_sys}.

We then form the Hankel matrix by stacking the elements of $\mathcal{H}$ as
\begin{equation}
    H \triangleq \begin{bmatrix} C B & C A^P B & \ldots & C A^{m_c P} B \\ C A^P B & C A^{2P} B & \ldots & C A^{(m_c +1) P} B \\ \vdots & \vdots & \ddots & \vdots \\ C A^{m_o P} B & C A^{(m_o + 1)P} B & \ldots & C A^{(m_c + m_o)P} B \end{bmatrix},
    \label{eq:def-Hankel-1}
\end{equation}
where $m_c$ and $m_o$ can be chosen arbitrarily, subject to $m_c + m_o = m - 1$.
The time-shifted Hankel matrix is built from the elements of $\mathcal{H}'$ in the same way:
\begin{equation}
    H' \triangleq \begin{bmatrix} C A B & C A^{P+1} B & \ldots & C A^{m_c P + 1} B \\ C A^{P+1} B & C A^{2P+1} B & \ldots & C A^{(m_c + 1) P + 1} B \\ \vdots & \vdots & \ddots & \vdots \\ C A^{m_o P + 1} B & C A^{(m_o + 1)P + 1} B & \ldots & C A^{(m_c + m_o)P + 1} B
    \end{bmatrix}.
    \label{eq:def-Hankel-2}
\end{equation}

Next, we compute the (reduced) SVD of $H$, giving us
\begin{equation*}
    H = U \Sigma V^*.
\end{equation*}
Let $U_r$ consist of the first $r$ columns of $U$.
Similarly, let $\Sigma_r$ be the upper left $r \times r$ submatrix of $\Sigma$ and let $V_r$ contain the first $r$ columns of $V$.
Then the $r$-dimensional ERA approximation of~\eqref{eq:era_sys} is given by the reduced-order system
\begin{equation} \begin{aligned}
    \xi_{k+1} &= A_r \xi_k + B_r u_k \\
    \eta_k &= C_r \xi_k + D_r u_k,
    \label{eq:era-reduced-sys}
\end{aligned} \end{equation}
where
\begin{align}
    A_r &= \Sigma_r^{-1/2} U_r^* H' V_r \Sigma_r^{-1/2} \label{eq:def-Ar}\\
    B_r &= \text{the first $p$ columns of } \Sigma_r^{1/2} V_r^* \notag \\
    C_r &= \text{the first $q$ rows of } U_r \Sigma_r^{1/2} \notag \\
    D_r &= D \notag.
\end{align}

Now, suppose we use $H$ and $H'$ to compute DMD modes and eigenvalues as in Algorithm~\ref{alg:exact-DMD}, with $X=H$ and $Y=H'$.
Recall from~\eqref{eq:low-order-eig-prob} that the DMD eigenvalues and modes are determined by the eigendecomposition of the operator
\begin{equation*}
    \tilde{A} = U^* H' V \Sigma^{-1},
\end{equation*}
with $U$, $\Sigma$, and $V$ defined as above.
Comparing with~\eqref{eq:def-Ar}, we see that if one keeps all of the modes from ERA (i.e., choosing $r$ equal to the rank of~$H$), then $A_r$ and~$\tilde A$ are related by a similarity transform, with
\begin{equation*}
    A_r = \Sigma^{-1/2}\tilde A \Sigma^{1/2},
\end{equation*}
so $A_r$ and~$\tilde{A}_r$ have the same eigenvalues.  Furthermore, if $v$ is an eigenvector of $A_r$, with
\begin{equation*}
    A_r v = \lambda v,
\end{equation*}
then $w=\Sigma^{1/2} v$ is an eigenvector of $\tilde A$, since
\begin{equation*}
  \tilde A w = \tilde A \Sigma^{1/2} v = \Sigma^{1/2} A_r v = \lambda\Sigma^{1/2} v = \lambda w.
\end{equation*}
Then $w$ can be used to construct either the exact or projected DMD modes (see~\eqref{eq:def-exact-DMD-mode} and~\eqref{eq:def-proj-DMD-mode}).

We see that algorithmically, the ERA and DMD are closely related: given two matrices $H$ and $H'$, applying the ERA produces a reduced-order operator $A_r$ that can be used to compute DMD modes and eigenvalues.
However, the ERA was originally developed to characterize linear input-output systems, for which~\eqref{eq:era-reduced-sys} approximates~\eqref{eq:era_sys}.
(In practice, it is often applied to experimental data, for which the underlying dynamics may be only approximately linear.)
On the other hand, DMD is designed to analyze data generated by any dynamical system; the system can be nonlinear and may have inputs (e.g., $x_{k+1} = f(x_k, u_k)$) or not (e.g., $x_{k+1} = f(x_k)$).

We note that in the ERA, we collect two sets of data $\mathcal{H}$ and $\mathcal{H}'$, then arrange that data in matrices $H$ and $H'$.
In doing so we are free to choose the values of $m_c$ and $m_o$, which determine the shapes of $H$ and $H'$ in~\eqref{eq:def-Hankel-1} and~\eqref{eq:def-Hankel-2}.
Interpreting DMD using the Koopman formalism of Section~\ref{ssec:Koopman}, each column of $H$ corresponds to a particular value of an observable, where the observable function is a vector of outputs at $m_o+1$ different timesteps.
Each column of the matrix~$H'$ then contains the value of this observable at the next timestep.
A more typical application of DMD would use $m_o = 0$ and $m_c = m - 1$.
Allowing for $m_o > 0$ is equivalent to appending the data matrices with rows of time-shifted Markov parameters.
Doing so can enlarge the rank of $H$ and $H'$, increasing the accuracy of the reduced-order system~\eqref{eq:era-reduced-sys} for the ERA.
For DMD, it can overcome the rank limitations that prevent the correct characterization of standing waves (as suggested in Section~\ref{sssec:standing-wave}).

\subsection{Linear inverse modeling (LIM)}
\label{ssec:LIM}

In this section, we investigate the connections between DMD and LIM.
To set up this discussion, we briefly introduce and define a number of terms used in the climate science literature.
This is followed by a more in-depth description of LIM.
Finally, we show that under certain conditions, LIM and DMD are equivalent.

\subsubsection{Nomenclature}

\emph{Empirical orthogonal functions} (EOFs) were first introduced in 1956 by Lorenz~\cite{lorenzMITTR56}.
Unique to the climate science literature, EOFs simply arise from the application of principal component analysis (PCA)~\cite{pearsonPhilMag01,hotellingJEdPsy33_1,hotellingJEdPsy33_2} to meteorological data~\cite{penlandMWR89}.
As a result, EOF analysis is equivalent to PCA, and thus also to POD and SVD.
(We note that in PCA and and EOF analysis, the data mean is always subtracted, so that the results can be interpreted in terms of variances; this is often done for POD as well.)

In practice, EOFs are often used as a particular choice of \emph{principal interaction patterns} (PIPs), a concept introduced in 1988 by Hasselmann~\cite{hasselmannJGRA88}.
The following discussion uses notation similar to that found in~\cite{vonstorchJClimate95}, which provides a nice review of PIP concepts.
Consider a dynamical system with a high-dimensional state $x(t) \in \mathbb{R}^n$.
In some cases, such a system may be approximately driven by a lower-dimensional system with state $z(t) \in \mathbb{R}^r$, where $r < n$.
To be precise, we say that $x$ and $z$ are related as follows:
\begin{align*}
    z_{k+1} &= F(z_k; \alpha) + \text{noise} \\
    x_k &= Pz_k + \text{noise},
\end{align*}
where $\alpha$ is a vector of parameters.
From this we see that given a knowledge of $z$ and its dynamics, $x$ is completely specified by the static map $P$, aside from the effects of noise.
Though in general $P$ cannot be inverted, given a measurement of $x$, we can approximate $z$ using a least-squares fit:
\begin{equation*}
    z_k = (P^T P)^{-1}P^T x_k.
\end{equation*}

In climate science, the general approach of modeling the dynamics of a high-dimensional variable $x$ through a lower-dimensional variable $z$ is referred to as \emph{inverse modeling}.
The inverse model described above requires definitions of $F$, $P$, and $\alpha$.
Generally, $F$ is chosen based on physical intuition.
Once that choice is made, $P$ and $\alpha$ are fitted simultaneously.
The PIPs are the columns of $P$ for the choice of $P$ (and $\alpha$) that minimizes the error
\begin{equation*}
    \epsilon(P, \alpha) \triangleq E\bigg( \Big\|x_{k+1} - x_k - P\Big(F(z_k; \alpha) - z_k\Big) \Big\| \bigg),
\end{equation*}
where $E$ is the expected value operator~\cite{vonstorchJClimate95}.
In general, the choice of $P$ is not unique.

Hasselmann also introduced the notion of \emph{principal oscillation patterns} (POPs) in his 1988 paper~\cite{hasselmannJGRA88}.
Again, we use the notation found in~\cite{vonstorchJClimate95}.
Consider a system with unknown dynamics.
We assume that we can approximate these dynamics with a linear system
\begin{equation*}
    x_{k+1} = A x_k + \text{noise}.
\end{equation*}
If we multiply both sides by $x_k^T$ and take expected values, we can solve for $A$ as
\begin{equation}
    A = E(x_{k+1} x_k^T) E(x_k x_k^T)^{-1}.
    \label{eq:POP_operator}
\end{equation}
The eigenvectors of $A$ are referred to as POPs.
That is, POPs are eigenvectors of a particular linear approximation of otherwise unknown dynamics.

Even within the climate science literature, there is some confusion between PIPs and POPs.
This is due to the fact that POPs can be considered a special case of PIPs.
In general, PIPs are basis vectors spanning a low-dimensional subspace useful for reduced-order modeling.
Suppose we model our dynamics with the linear approximation described above, and do not reduce the order of the state.
If we then express the model in its eigenvector basis, we are choosing our PIPs to be POPs.

\subsubsection{Linear Markov models/linear inverse modeling (LIM)}

In 1989, Penland derived a method for computing a linear, discrete-time system that approximates the trajectory of a stochastic, continuous-time, linear system, which he referred to as a linear Markov model~\cite{penlandMWR89}.
We describe this method, which came to be known as LIM, using the notation found in~\cite{penlandJClimate93}.
Consider an $n$-dimensional Markov process
\begin{equation*}
    \frac{dx}{dt} = B x(t) + \xi(t),
\end{equation*}
where $\xi(t)$ is white noise with covariance
\begin{equation*}
    Q = E\left(\xi(t) \xi^T(t) \right).
\end{equation*}
We assume the mean of the process has been removed.
The covariance of $x$ is given by
\begin{equation}
    \label{eq:LIM_Lambda}
    \Lambda = E\left(x(t) x^T(t)\right).
\end{equation}
One can show that the following must hold:
\begin{gather*}
    B \Lambda + \Lambda B^T + Q = 0 \\
    E\left(x(t+\tau)x^T(t)\right) = \exp(B\tau) \Lambda.
\end{gather*}
(See~\cite{penlandMWR89} for details.)

Defining the Green's function
\begin{align}
    \label{eq:LIM_G}
    G(\tau) &\triangleq \exp(B\tau) \notag \\
    &= E\left(x(t+\tau)x^T(t)\right) \Lambda^{-1},
\end{align}
we can say that given a state $x(t)$, the most probable state time $\tau$ later is
\begin{equation*}
    x(t+\tau) = G(\tau) x(t).
\end{equation*}
The operator $G(\tau)$ is computed from snapshots of the continuous-time system and has the same form as the linear approximation used in POP analysis (see~\eqref{eq:POP_operator}).
We note that we arrive at the same model if we apply linear stochastic estimation to snapshots of the state $x$, taking $x(t)$ and $x(t+\tau)$ to be the unconditional and conditional variables, respectively.
(This was done in~\cite{tuEF13} to identify a model for the evolution of POD coefficients in a fluid flow.)

When this approach is applied to a nonlinear system, it can be shown that $G(\tau)$ is equivalent to a weighted average of the nonlinear dynamics, evaluated over an ensemble of snapshots~\cite{blumenthalJClimate91}.
This is in contrast to a typical linearization, which involves evaluating the Jacobian of the dynamics at a fixed point.
If the true dynamics are nearly linear, these two approaches will yield nearly the same model.
However, if nonlinear effects are significant, $G(\tau)$ will be closer to the ensemble average, and arguably a better model than a traditional linearization~\cite{blumenthalJClimate91}.

In~\cite{penlandJClimate93}, this method was applied to compute a linear Markov model in the space of EOF coefficients.
This is an example of inverse modeling (equivalently, PIP analysis); a high-dimensional variable is modeled via a projection onto a lower-dimensional EOF subspace.
Due to the assumption of linear dynamics, this approach came to be known as \emph{linear inverse modeling}.
The combination of PIP and POP concepts in this early work has contributed to the continuing confusion between PIPs and POPs in the climate science literature today.

\subsubsection{Equivalence to projected DMD}

In both exact and projected DMD, the DMD eigenvalues are given by the eigenvalues of the projected linear operator $\tilde{A}$ (see~\eqref{eq:def-Atilde}).
The projected DMD modes are computed by lifting the eigenvectors of $\tilde{A}$ to the original space via the left singular vectors $U$ (see~\eqref{eq:def-proj-DMD-mode}).
In~\cite{penlandJClimate93}, the eigendecomposition of a low-order linear model $G(\tau)$ is computed and the low-order eigenvectors lifted to the original space via EOFs, in the same way as in projected DMD.
The similarity in these two approaches is obvious.
Recall that left singular vectors and EOFs are equivalent, so long as they are computed from the same data.
Then to prove that LIM-based eigenvector analysis is equivalent to projected DMD, we simply have to show the equivalence of $G(\tau)$ and $\tilde{A}$.

Consider two $n \times m$ data matrices $X$ and $Y$, with columns $x_j = x(t_j$) and $y_j = x(t_j + \tau)$, respectively.
$X$ and $Y$ may or may not share columns.
As in~\eqref{svd-of-X}, we assume that the EOFs to be used for LIM are computed from $X$ alone, giving us
\begin{equation*}
    X = U \Sigma V^*,
\end{equation*}
where the columns of $U$ are the EOFs.
The EOF coefficients of $X$ and $Y$ are given by
\begin{equation}
    \label{eq:def_Xhat}
    \hat{X} = U^* X, \qquad \hat{Y} = U^* Y,
\end{equation}
whose columns we donote by $\hat{x}_j$ and $\hat{y}_j$, respectively.


In order to show that $G(\tau)$ and $\tilde{A}$ are equivalent, we must reduce~\eqref{eq:LIM_G} to~\eqref{eq:def-Atilde}.
Because we are interested in the equivalence of LIM and projected DMD when the former is performed in the space of EOF coefficients, we replace all instances of $x$ in~\eqref{eq:LIM_Lambda} and~\eqref{eq:LIM_G} with $\hat{x}$.
Recall that the expected value of $a$, for an ensemble $\{a_j\}_{j=0}^{m-1}$, is given by
\begin{equation}
    E\left(a_j\right) \triangleq \frac{1}{m}\sum_{j=0}^{m-1} a_j.
\end{equation}
Then we can rewrite~\eqref{eq:LIM_Lambda} as
\begin{align*}
    \Lambda &= \frac{1}{m} \sum_{j=0}^{m-1} \hat{x}_j \hat{x}_j^* \\
    &= \frac{1}{m} \hat{X} \hat{X}^* \\
    &= \frac{1}{m} U^* X X^* U \\
    &= \frac{1}{m} U^* U \Sigma^2 \\
    &= \frac{1}{m} \Sigma^2,
\end{align*}
using the fact that $X X^* U = U \Sigma^2$, by the definition of left singular vectors.
This result, along with~\eqref{svd-of-X}, allows us to rewrite~\eqref{eq:LIM_G} as
\begin{align*}
    G(\tau) &= \left(\frac{1}{m} \sum_{j=0}^{m-1} \hat{y}_j \hat{x}_j^*\right) \left(m \Sigma^{-2}\right) \notag \\
    &= \hat{Y} \hat{X}^* \Sigma^{-2} \\
    &= U^* Y X^* U \Sigma^{-2} \\
    &= U^* Y V \Sigma U^* U \Sigma^{-2} \\
    &= U^* Y V \Sigma^{-1}.
\end{align*}
(Recall that $x(t_j+\tau) = y_j$, and $x(t_j) = x_j$.)
From~\eqref{eq:def-Atilde}, we then have $G(\tau) = \tilde{A}$, and we see that DMD and LIM are built on the same low-dimensional, approximating linear dynamics.

We emphasize that this equivalence relies on a number of assumptions.
First, we assume that we perform LIM in the space of EOF coefficients.
Second, we assume that the EOFs are computed from $X$ alone.
This may not be an intuitive choice if $X$ and $Y$ are completely distinct, but for a sequential snapshot sequence where $X$ and $Y$ differ by a single column, this is not a significant difference.

Given these assumptions, the equivalence of projected DMD and LIM gives us yet another way to interpret DMD analysis.
If the data mean is removed, then the low-order map that generates the DMD eigenvalues and eigenvectors is simply the one that yields the statistically most likely state in the future.
(This is the case for both exact and projected DMD, as both are built on the same low-order linear map.)
In a small sense, the DMD framework is more general, as the intrepretation provided by Definition~\ref{def:exact-DMD} holds even for data that are not mean-subtracted.
Then again, in LIM the computation of the EOFs is completely divorced from the modeling procedure, allowing for a computation using both $X$ and $Y$.
Nevertheless, the similarities between the two methods suggests that practitioners of DMD would be well-served in studying and learning from the climate science/LIM literature.


\section{Conclusions}
\label{sec:concl}

We have presented a new definition in which DMD is defined to be the eigendecomposition of an approximating linear operator.
Whereas existing DMD theory focuses on full-rank, sequential time series, our theory applies more generally to pairs of data vectors.
At the same time, our DMD algorithm is only a slight modification of the commonly used, SVD-based DMD algorithm.
It also preserves, and even strengthens, the links between DMD and Koopman operator theory.
Thus our framework can be considered to be an extension of existing DMD theory to a more general class of datasets.

For instance, when analyzing data generated by a dynamical system, we require only that the columns of the data matrices $X$ and $Y$ be related by the dynamics of interest, in a pairwise fashion.
Unlike existing DMD algorithms, we do not require that the data come from uniform sampling of a single time series, nor do we require that the columns of $X$ and $Y$ overlap.
We demonstrated the utility of this approach using two numerical examples.
In the first, we sampled a trajectory nonuniformly, significantly reducing computational costs.
In the second, we concatenated multiple datasets in a single DMD computation, effectively averaging out the effects of noise.
Our generalized interpretation of DMD also proved useful in explaining the results of DMD computations involving rank-deficient datasets.
Such computations may provide either meaningful or misleading information, depending on the dataset, and are not treated in the existing DMD literature.

Finally, we showed that DMD is closely related to both the eigensystem realization algorithm (ERA) and linear inverse modeling (LIM).
In fact, under certain conditions DMD and LIM are equivalent.
We used the connection between DMD and the ERA to motivate a strategy for dealing with the inability of DMD to correctly characterize standing waves.
An interesting future direction would be to explore whether or not lessons learned from past applications of LIM can similarly inform strategies for future applications of DMD.


\section*{Acknowledgments}
The authors gratefully acknowledge support from the Air Force Office of Scientific Research (AFOSR), the National Science Foundation (NSF), and the Federal Aviation Administration (FAA).
In addition, the authors thank Peter J. Schmid, Shervin Bagheri, Matthew O. Williams, Kevin K. Chen, and Jay Y. Qi for their valuable insights.
Specifically, the authors acknowledge Jay Y. Qi for first observing the limitations of DMD with regard to standing wave behavior, during his study of the Kuramoto--Sivashinsky equation.
Experimental data from the bluff-body wake discussed in Section~\ref{sssec:comb-mult-traj} were collected by John Griffin, Adam Hart, Louis N. Cattafesta III, and Lawrence S. Ukeiley.


\appendix

\section{Choices for mode scaling}
\label{sec:mode-scalings}

DMD practitioners often identify DMD modes of interest using their norms; it is generally assumed that modes with large norms are more dynamically important.
Sometimes, as in Figure~\ref{fig:bluff_body_wake_dmd_spectra}, the norms are weighted by the magnitude of the corresponding DMD eigenvalues to penalize spurious modes with large norms but quickly decaying contributions to the dynamics.
However, Definition~\ref{def:exact-DMD} allows for arbitrary scalings of the DMD modes.
Then in order to use mode norms as a criterion for selecting interesting DMD modes, an appropriate scaling must be chosen.
(We note that a nice approach to scaling was developed in~\cite{jovanovicArxiv13}, which considers an sparse representation where only a small number of DMD modes have nonzero magnitude.
We do not pursue this approach here, but the ideas in~\cite{jovanovicArxiv13} apply equally well to exact DMD modes and to projected DMD modes.)

Below, we discuss a number of possible options for scaling, any of which could be appropriate depending on the application.
Our discussion is split into two parts: first, we consider general scalings motivated by linear algebra concepts.
Second, we consider scalings that are appropriate for sequential time series.
We pay particular attention to the cost of computing the scaling factors, as DMD is primarily applied to high-dimensional datasets for which data manipulations can be costly.
For instance, due to computational costs, in practice only a minimal set of modes is computed.
This prohibits any scaling that requires first computing all of the DMD modes.

\subsection{Generic data}

Perhaps the most obvious scaling for an eigenvector is one such that each eigenvector has unit norm.
This can certainly be applied in computing DMD modes, and such scaling factors can be computed on a mode-by-mode basis.
For instance, one may first normalize the eigenvectors of $\tilde{A}$ ($w$ in~\eqref{eq:low-order-eig-prob}).
Then the projected DMD modes $\hat\phi = \mathbb{P}_X\phi = Uw$ also have unit norm, since $\hat\phi^*\hat\phi=w^*U^*Uw=w^*w=1$.
The scalings in~\eqref{eq:def-proj-DMD-mode} and~\eqref{eq:def-exact-DMD-mode} are then natural, recalling from Theorem~\ref{thm:proj-DMD-modes} that $\hat \phi = \mathbb{P}_X \phi$.

Another approach is to consider the relationship between the (exact) DMD modes and the adjoint DMD modes.
Recall that adjoint DMD modes are constructed via $\psi \triangleq U z$, where $z^*\tilde{A} = \lambda z^*$.
We observe that exact DMD modes are orthogonal to adjoint DMD modes with different eigenvalues, since
\begin{gather*}
    (\lambda_j\psi_j^*)\phi_k = \psi_j^* A \phi_k = \psi_j^*(\lambda_k\phi_k)\\
    \implies (\lambda_j-\lambda_k) \psi_j^* \phi_k = 0.
\end{gather*}
So as long as $\lambda_j\ne\lambda_k$, then $\psi_j^* \phi_k$ is zero.
If we scale the adjoint eigenvectors $z$ such that $z^* w = 1$, leaving $\|w\| = 1$, we observe that
\begin{align*}
    \psi_k^* \phi_k &= \frac{1}{\lambda} (U z_k)^* (Y V \Sigma^{-1} w_k) \\
    &= \frac{1}{\lambda} z_k^* \tilde{A} w_k  = 1.
\end{align*}
Thus, the exact DMD modes and adjoint DMD modes form a biorthogonal set; there is no additional scaling necessary after computing the DMD modes via~\eqref{eq:def-exact-DMD-mode}.
(We note that it is also possible to let $\|z\| = 1$ and scale $w$ to achieve the above result.)

\subsection{Sequential time series}

Based on~\eqref{eq:Koopman_dyn_seq}, projected DMD modes are often scaled such that the sum of the modes equals the first data vector, $x_0$~\cite{chenJNLS12}.
This is also natural for exact DMD modes.
However, representing $x_0$ using eigenvectors of $A$ generally requires eigenvectors whose eigenvalues are zero; these modes are not identified by Algorithm~\ref{alg:exact-DMD}.
Fortunately, it turns out that we can compute a scaling based on $x_0$ without computing those modes.

Let $\Phi = Y V \Sigma^{-1} W \Lambda^{-1}$ be the matrix whose columns are the DMD modes computed by Algorithm~\ref{alg:exact-DMD} (those with nonzero DMD eigenvalues), where $\tilde{A} W = W \Lambda$.
In order to satisfy~\eqref{eq:Koopman_dyn_seq} (for $k=1$), we must have
\begin{equation}
    \label{eq:sum-modes-to-y0}
    \Phi \Lambda d = y_0,
\end{equation}
where $d=(d_1,\ldots,d_j)$ with $d_j=\theta_j(z)$ in~\eqref{eq:Koopman_dyn_seq}.  We wish to solve~\eqref{eq:sum-modes-to-y0} for~$d$.
Note that, as long as the data matrices $X$ and~$Y$ are linearly consistent (Definition~\ref{def:linear-consistency}), we have $y_0=Ax_0$, so $y_0$ is in the range of~$A$, and therefore in the range of~$\Phi$, and~\eqref{eq:sum-modes-to-y0} has a solution.
This solution may be found, for instance, by QR decomposition of~$\Phi$.
Then $d_j$ contains the ``amplitude'' of each mode $\phi_j$, and if desired, the modes may be renormalized by these amplitudes.
(If no exact solution exists, we can instead solve for the least-squares solution.)

For large problems, it can be costly to compute the full set of DMD modes~$\Phi$.
In this case, one may solve for~$d$ using a pseudoinverse of~$\Phi$, following~\cite{belsonACMTMS13}:
\begin{align*}
    d &= \Lambda^{-1} (\Phi^* \Phi)^{-1} \Phi^* y_0 \\
    &= (W^* \Sigma^{-1} V^* Y^* Y V \Sigma^{-1} W)^{-1} W^* \Sigma^{-1} V^* Y^* y_0.
\end{align*}
In this approach, one does not need to explicitly calculate~$\Phi$; however, one does need the product $Y^* Y$, which is expensive to compute for large $n$.
Consider that when $n \gg m$, the most expensive step in the DMD algorithm is computing $X^* X$, required to compute the SVD of $X$ using the method of snapshots~\cite{sirovichQAM87_2}.
Then for generic $Y$, computing scaling factors from $Y^* Y$ is nearly as costly as the rest of the DMD algorithm, and may not be viable.
Furthermore, in this approach, the condition number of~$\Phi$ is squared, so for poorly conditioned matrices, the solution can be much less accurate than alternatives such as QR decomposition.
If, however, $X$ and $Y$ describe a sequential time-series (see~\eqref{eq:def-XY-sequential}), then $Y^* Y$ and $X^* X$ will contain many of the same elements.
In this case, $X^* X$ has already been computed and computing $Y^* Y$ requires only $m$ new inner products.


\bibliographystyle{unsrt}
\bibliography{jabbrv,references}




\end{document}